\documentclass{amsart}
\usepackage{amssymb}
\usepackage{amscd}
\usepackage{amsfonts}
\usepackage{latexsym}
\usepackage[all]{xy}
\usepackage{verbatim}
\usepackage{hyperref}
\usepackage{mathenv}
\usepackage{mathbbol}
\usepackage{mathrsfs}
\usepackage{tikz}

\newtheorem{theorem}{Theorem}[section]
\newtheorem{lemma}[theorem]{Lemma}
\newtheorem{proposition}[theorem]{Proposition}
\newtheorem{corollary}[theorem]{Corollary}
\theoremstyle{definition}
\newtheorem{definition}[theorem]{Definition}
\newtheorem{example}[theorem]{Example}
\newtheorem{question}[theorem]{Question}
\newtheorem{conjecture}[theorem]{Conjecture}
\newtheorem{remark}[theorem]{Remark}

\newcommand{\R}{\mathbb{R}}
\newcommand{\G}{\mathbb{G}}
\newcommand{\C}{\mathcal{C}}
\newcommand{\D}{\mathcal{D}}
\newcommand{\W}{\mathcal{W}}
\newcommand{\Sl}{\mathcal{S}}
\newcommand{\BC}{\mathbb{C}}
\newcommand{\Z}{\mathbb{Z}}
\newcommand{\K}{\mathbb{K}}

\newcommand{\g}{\mathfrak{g}}

\newcommand{\cZ}{\mathcal{Z}}
\newcommand{\A}{\mathcal{A}}
\newcommand{\B}{\mathcal{B}}
\newcommand{\T}{\mathcal{T}}
\newcommand{\cI}{\mathcal{I}}
\newcommand{\mO}{\mathcal{O}}

\newcommand{\FP}{\text{FPdim}}
\newcommand{\Rep}{\text{Rep}}
\newcommand{\Irr}{\text{Irr}}
\newcommand{\Id}{\text{Id}}
\newcommand{\Hom}{\mathrm{Hom}}
\newcommand{\Ext}{\text{Ext}}
\newcommand{\End}{\text{End}}
\newcommand{\be}{{\bf 1}}
\numberwithin{equation}{section} 
\usepackage{color}

\begin{document}
\title{A non-semisimple Witt class}

\author{Victor Ostrik}
\address{V.O.: Department of Mathematics,
University of Oregon, Eugene, OR 97403, USA}
\email{vostrik@math.uoregon.edu}

\author{Alexandra Utiralova}
\address{V.O.: Department of Mathematics,
University of Oregon, Eugene, OR 97403, USA}
\email{auti@math.uoregon.edu}

\begin{abstract} We describe several infinite families of braided finite tensor categories.
A simplest example gives a non-degenerate braided tensor category which is not Witt equivalent
to a semisimple category.
\end{abstract}

\date{\today} 
\maketitle  

\section{Introduction}
\subsection{}
This paper is a contribution to the theory of braided finite tensor categories. In the case
of semisimple categories over $\BC$ many known examples arise from Wess-Zumino-Witten
models in conformal field theory, see e.g. \cite{BK}. One of the easiest algebraic constructions of these categories was given
by H.~H.~Andersen \cite{An} (see also \cite{Saw}). In modern terms this construction can be summarized
as follows: for a simple Lie algebra $\g$ and a root of unity $q$ such that $l$ is the order of $q^2$ one considers
the category $\T(\g,q)$ of {\em tilting modules} over the quantum group at a root of unity $q$ associated 
with Lie algebra $\g$. Then one defines a category $\C(\g, l, q)$ as the {\em semisimplification}
of   $\T(\g,q)$ (see e.g. \cite{EOsemi}). The category $\C(\g, l, q)$ is a semisimple braided tensor
category; moreover this category is finite (i.e. it has only finitely many classes of simple objects) if
$l$ is sufficiently large. The precise bounds for $l$ are given in \cite[Figure 2]{Scho}; in this paper 
we will always assume that $l$ is sufficiently large in this sense. Note that there is a combinatorial 
difference between the case when $l$ is {\em divisible} (see \ref{divisible}) and when it is not. 

The procedure of semisimplification above can be described as taking the quotient by a suitable
tensor ideal (namely, by the ideal of negligible morphisms). However, the category $\T(\g,q)$ has 
many other tensor ideals. In \cite{CEO} for any {\em distinguished} nilpotent element $e\in \g$ (or, in the case
when $l$ is divisible, $e\in \g^L$ where $\g^L$ is the Langlands dual Lie algebra of $\g$),
a tensor ideal $\cI_e$ was constructed such that the quotient category $\T(\g,q)/\cI_e$ admits a {\em monoidal
abelian envelope} $\C(\g,e,l,q)$; moreover, the category $\C(\g,e,l,q)$ is a finite tensor category
in the sense of \cite{EOfinite}. For example when $e$ is a regular nilpotent element (which is always
distinguished), the ideal $\cI_e$ is the ideal of negligible morphisms and $\C(\g,e,l,q)=\C(\g,l,q)$.
Unfortunately, there is not much we can say about the category $\C(\g,e,l,q)$ for other nilpotent
elements $e$ (however, see Section \ref{distconj} for a conjectural formula for the Frobenius-Perron dimension
of $\C(\g,e,l,q)$ and conjectural description of its cohomology). The goal of this paper is to give 
some explicit information about the category $\C(\g,e,l,q)$ in the case when $e=e_{sr}$ is a subregular
nilpotent element, see \cite[4.2]{CMg}. Recall that $e_{sr}$ is distinguished if and only if $\g$ is not of
types $A_n, n\ge 1$ or $B_n, n\ge 2$, see \cite{CMg}.

To state our main results we need a bit more notation. Let $V$ be a two dimensional space and let
$\Gamma \subset SL(V)$ be a finite subgroup of even order. By the classical McKay correspondence,
the finite subgroups of $SL(V)$ up to conjugacy are labelled by simply laced affine Dynkin diagrams (the only
subgroups of odd order correspond to affine Dynkin diagrams of type $\tilde A_{2n}$). We associate
a subgroup $\Gamma$ to $\g$ and $l$ as above as follows:
\renewcommand{\arraystretch}{1.4}
$$
\begin{array}{|c|c|c|c|c|c|c|}
\hline
\g&B_n&C_n&D_n&E_n&F_4&G_2\\
&n\ge 3&n\ge 3&n\ge 4&n=6,7,8&&\\
\hline
l&\mbox{divisible}&\mbox{not divisible}&\mbox{any}&\mbox{any}&\mbox{any}&\mbox{any}\\
\hline
\Gamma &\tilde D_{2n} &\tilde D_{2n}&\tilde D_n&\tilde E_n&\tilde E_7&\tilde E_7\\
\hline
\end{array}
$$

Let $\wedge(V)$ be the exterior algebra of $V$. We can consider $\wedge(V)$ as an algebra in the category
$\mbox{Rep}(\Gamma)$ of finite dimensional representations of $\Gamma$. We consider the abelian
category of left $\wedge(V)-$modules in the category $\mbox{Rep}(\Gamma)$ and we call it
{\em block of type} $\Gamma$. Thus, a block of type $\Gamma$ is the category equivalent to that of finite dimensional
representations of the cross product of $\wedge(V)$ with the group algebra of $\Gamma$.
Finally let $S^\bullet(V)$ be the symmetric algebra of $V$ which is graded by even integers (so $V \subset
S^\bullet(V)$ is in degree 2). The group $\Gamma$ acts on $S^\bullet(V)$ preserving the grading.
Our first main result describes the structure of abelian category $\C(\g,e_{sr},l,q)$.

\begin{theorem}\label{main blocks}
The category $\C(\g,e_{sr},l,q)$ decomposes into blocks which are either trivial
(that is equivalent to the category of vector spaces) or of type $\Gamma$. In particular,
the category $\C(\g,e_{sr},l,q)$ is of tame representation type. Also the cohomology of
$\C(\g,e_{sr},l,q)$ (that is the Ext algebra of the unit object $\be$) is isomorphic to the algebra
of invariants $S^\bullet(V)^\Gamma$.
\end{theorem}

\begin{remark} The proof of Theorem \ref{main blocks} shows that the number of blocks of type $\Gamma$
is the same as the number of weights inside of the fundamental alcove (but not on its boundary). Thus, this
number is the same as the number of simple objects in the category $\C(\g,l,q)$.
\end{remark}

\subsection{} Next, we study one specific example, the category $\C(G_2,G_2(a_1),7,q)$ (thus, we
consider Lie algebra $\g$ of type $G_2$; also $G_2(a_1)$ is the standard notation for the subregular
nilpotent orbit in type $G_2$). This is the simplest example of the categories considered above
(at least for undivisible $l$). Recall the standard notation for the quantum numbers:
$$[k]_l=\frac{\sin(k\pi/l)}{\sin(\pi/l)}.$$
In particular, $[2]_7=[5]_7=2\cos(\pi/7)\approx 1.801938$ and $[3]_7=[4]_7=\frac{\sin(3\pi/7)}{\sin(\pi/7)}\approx
2.246980$.

\begin{theorem}\label{main 7}
(1) The category $\C(G_2,G_2(a_1),7,q)$ has 15 trivial blocks and one block of type $\tilde E_7$.
In particular it has 23 simple objects.

(2) We have $\FP(\C(G_2,G_2(a_1),7,q))=294(7+15[3]_7+12[5]_7)\approx 18324.416384$.

(3) The category $\C(G_2,G_2(a_1),7,q)$ has stable Chevalley property: tensor products of simple 
objects are direct sums of simples and projectives.

(4) The M\"uger center of the category $\C(G_2,G_2(a_1),7,q)$ is equivalent to $\Rep(S_3)$
(where $S_3$ is the symmetric group on three letters).
\end{theorem}

In view of Theorem \ref{main 7} (4), it makes sense to consider the de-equivariantization $\bar \C(G_2,G_2(a_1),7,q)$ 
of $\C(G_2,G_2(a_1),7,q)$ with respect to its M\"uger center (so the category $\bar \C(G_2,G_2(a_1),7,q)$ 
is the Brugui\`eres' modularisation of $\C(G_2,G_2(a_1),7,q)$, see \cite{Brug}). The category
$\bar \C(G_2,G_2(a_1),7,q)$ is a non-semisimple modular tensor category in the sense of Shimizu,
see \cite{Shim}.

\begin{theorem}\label{main 7mod}
(1) The category $\bar \C(G_2,G_2(a_1),7,q)$ has 12 trivial blocks and one block of type $\tilde D_4$.
In particular, it has 17 simple objects.

(2) We have $\FP(\bar \C(G_2,G_2(a_1),7,q))=49(7+15[3]_7+12[5]_7)\approx 3054.068811$.

(3) The category $\bar \C(G_2,G_2(a_1),7,q)$ has stable Chevalley property.

(4) The category $\bar \C(G_2,G_2(a_1),7,q)$ is completely anisotropic: it has no non-trivial
commutative exact algebras.
\end{theorem}

In \cite[Question 7.20]{ShYa} (see also \cite[Question 6.25]{LaWa}) K.~Shimizu and H.~Yadav asked whether non-semisimple completely
anisotropic categories exist;  Theorem \ref{main 7mod} (4) gives a positive answer to this question.

In \cite[Definition 6.23]{LaWa} (see also \cite[Definition 7.2]{ShYa}) R.~Laugwitz and C.~Walton defined an important Witt equivalence
relation on the set of non-degenerate braided finite tensor categories. In Section \ref{Witt} we prove
some general properties of this relation which imply

\begin{theorem}\label{main 7Witt}
The category $\bar \C(G_2,G_2(a_1),7,q)$ is not Witt equivalent to any semisimple category.
\end{theorem}

Thus, the non-semisimple Witt group is different from its semisimple version studied in \cite{DMNO}.

\subsection{} In Section \ref{genconj} we present some conjectures. Most importantly, we expect that the categories
$\C(\g,e,l,q)$  make sense for all nilpotent elements $e\in \g$ (or $e\in \g^L$
in the divisible case). If $e$ is not distinguished,
the category $\C(\g,e,l,q)$ is not finite; however we expect that it is obtained from a finite tensor
category by equivariantization. 

\subsection{Acknowledgements} Some ideas that led to this paper were inspired by participation of one of us (V.O.) in American Institute of Mathematics SQuaRE ``Lie algebras in symmetric tensor categories''; we are very grateful to this institution and to the fellow participants Iv\'an Angiono, Agustina Czenky, Pavel Etingof, Julia Plavnik, and Guillermo Sanmarco. We also thank Kevin Coulembier, Michael Finkelberg, Dmitri Nikshych and Kenichi Shimizu for useful discussions.

\section{Categories $\C(\g,e,l,q)$}
\subsection{Notations} We will use standard notions from representation theory of simple Lie algebras,
see e.g. \cite{Scho}.
Let $\g$ be a simple Lie algebra. Let $\Lambda$ be the weight lattice of $\g$ and
let $\Lambda^+\subset \Lambda$ be the set of  dominant weights. Let $W$ be the Weyl group
and let $\langle , \rangle$ be $W-$invariant scalar product on $\Lambda$ normalized 
by the condition $\langle \alpha, \alpha \rangle =2$ for a short root $\alpha$. Let $\rho \in \Lambda$ be
the half sum of the positive roots; we will often use the dot-action of $W$ given by 
$w\cdot \lambda =w(\lambda +\rho)-\rho$. For any $\lambda \in \Lambda^+$ we will denote
by $\chi_\lambda$ the character of irreducible $\g-$module with highest weight $\lambda$; thus 
$\chi_\lambda$ is an element of the group ring $\Z[\Lambda]$ given by the Weyl character formula.
We will use standard order relation on $\Lambda$: $\lambda \le \mu$ if $\mu -\lambda$ is a sum 
(possibly empty) of positive roots.

\subsection{Tilting modules for quantum groups}\label{qgtitling}
Let $m$ be the ratio of squared lengths of long and short roots. Thus $m=1$ for types $ADE$, $m=2$
for types $BCF$, and $m=3$ for type $G_2$. 
\begin{definition} \label{divisible}
We say that $l\in \Z_{\ge 1}$ is \textit{divisible} if $l$ is divisible
by $m$ and $l$ is \textit{undivisible} otherwise.     \end{definition}

Let $q\in \BC^\times$ be a root of unity such that the order of $q^2$ is $l$. We are going to consider
the category $\Rep(U_q)$ of finite dimensional representations of quantum group $U_q$ (with divided powers) associated with $\g$ 
and $q$. We refer the reader to \cite[Section 3]{AP} for precise definition of the category $\Rep(U_q)$ (where
the category $\Rep(U_q)$ appears in Section 3.19 and is denoted by $\mathscr C$). It follows from
the results of \cite[Chapter 32]{Lb} that $\Rep(U_q)$ has a natural structure of ribbon tensor category 
(as defined e.g. in \cite[8.10]{EGNO}). So the category $\Rep(U_q)$ is equipped with a braiding
$c_{X,Y}: X\otimes Y\simeq Y\otimes X$ and a ribbon structure (or twist) $\theta$ which is an automorphism
of the identity functor of $\Rep(U_q)$ satisfying 
\begin{equation}\label{twist}
    \theta_{X\otimes Y}=(\theta_X\otimes \theta_Y)\circ c_{Y,X}\circ c_{X,Y}, \; \; \theta_{X^*}=(\theta_X)^*.
\end{equation}

Let $\T(\g,q)$ be the full subcategory of $\Rep(U_q)$, consisting of tilting modules, see \cite[3.18]{AP}. We will not need a precise definition
of tilting modules, so we will list the properties which are important for this paper.

(1) The category $\T(\g,q)$ is closed under tensor products and duality; thus, it is a ribbon monoidal category.

(2) The category $\T(\g,q)$ is Karoubian (but not abelian); its indecomposable objects are labeled by
highest weights. So let $T(\lambda)\in \T(\g,q)$ be the indecomposable object with highest weight $\lambda \in \Lambda^+$.

(3) Since $T(\lambda)$ is indecomposable, any endomorphism of $T(\lambda)$  is a scalar plus a nilpotent endomorphism. In particular, this applies to the twist morphism $\theta_{T(\lambda)}$. Explicitly, 
$\theta_{T(\lambda)}-\theta_\lambda \Id$ is nilpotent where
$$\theta_\lambda =q^{\langle \lambda, \lambda +2\rho \rangle }.$$

(4) The character $\mbox{ch}(T(\lambda))$ is a positive combination of $\chi_\mu, \mu \in \Lambda^+$:
$$\mbox{ch}(T(\lambda))=\chi_\lambda +\sum_{\mu \in \Lambda^+, \mu < \lambda} a_{\lambda, \mu}\chi_\mu,
\; \; a_{\lambda, \mu}\in \Z_{\ge 0}.$$

(5) We can compute dimensions of Hom's between tilting modules using characters. Namely
for $T\in \T(\g,q)$ let $\hat T$ be a $\g-$module such that the character of $\hat T$ equals $\mbox{ch}(T)$.
Then for any $T, T' \in \T(\g,q)$ we have 
$$\dim \Hom_{\T(\g,q)}(T,T')=\dim \Hom_\g (\hat T, \hat T').$$

(6) The linkage principle controls when the numbers $a_{\lambda, \mu}$ from (4) are nonzero:
if $a_{\lambda, \mu}\ne 0$ then $\mu \in \tilde W_l\cdot \lambda$ where $\tilde W_l$ is a suitable
affine Weyl group acting via the dot-action, see \cite[3.17]{AP}. Note that the action of $\tilde W_l$ depends on $l$;
moreover there is a significant difference between the cases of divisible and undivisble $l$ (namely, two affine
Weyl groups appearing in these cases are dual to each other).

(7) We extend the action of $\tilde W_l$ to $\Lambda_\R:=\Lambda \otimes \R$. Let $\Lambda^+_\R \subset \Lambda_\R$ be the fundamental Weyl chamber (thus $\Lambda^+_\R$ consists of nonnegative
linear combinations of the fundamental weights).
The {\em fundamental alcove} is defined as
$$C_l(\g)=\left\{ \begin{array}{cl}\{ \lambda \in \Lambda^+_\R -\rho \;|\; \langle \lambda +\rho, \beta_0\rangle \le l\}&
\mbox{if}\; l\; \mbox{is divisible}\\
\{ \lambda \in \Lambda^+_\R -\rho \;|\; \langle \lambda +\rho, \beta_1\rangle \le l\}&
\mbox{if}\; l\; \mbox{is undivisible}\\
\end{array} \right. $$
where $\beta_0$ is the highest root and $\beta_1$ is the highest short root.
The alcove $C_l(\g)$ is a fundamental domain for the dot-action of $\tilde W_l$. For any
$w\in \tilde W_l$, we set $C_w=w\cdot C_l(\g)$, in particular $C_e=C_l(\g)$. 
The subsets $C_w$ are (closed) alcoves and the set $\Lambda_\R$ is tiled by $C_w, w\in \tilde W_l$.
Let $\tilde W_l^0\subset \tilde W_l$ be the subset of shortest representatives of right cosets of $W$ 
in $\tilde W_l$. Then $\Lambda^+_\R-\rho=\bigcup_{w\in \tilde W_l^0}C_w$. It is easy to see that for any
$w\in \tilde W_l$, the set $C_w\cap \Lambda$ is finite.

In this paper we will consider only the case when $C_l(\g)$ contains a weight from $\Lambda$ in its
interior (equivalently, zero weight has a trivial stabilizer with respect to $\tilde W_l$ dot-action). 
Thus, $l\ge mh^\vee$ in the divisible case and $l\ge h$ in the undivisible case, where $h$ and $h^\vee$ are
the Coxeter and dual Coxeter numbers for $\g$ (see e.g. \cite[Figure 2]{Scho} for explicit bounds).

\subsection{Tensor ideals}\label{t ideal}
We recall that the set $\tilde W_l^0$ is a disjoint union of {\em canonical right cells}, see \cite{LX}. 
These cells are naturally labeled by the nilpotent orbits of $\g$ in the undivisible case and of $\g^L$ in the
divisible case (this is combination of \cite[Theorem 1.2]{LX} and \cite[Theorem 4.8]{Lucells4}).
The cell $A_e$ corresponding to a nilpotent element $e\in \g$ (or $e\in \g^L$) is finite if and only if the element
$e$ is distinguished, see \cite[Theorem 8.1]{Lucells4}. 

There is an order relation $\le_R$ on the set of right cells (and hence on the set of canonical right cells).
Given a canonical right cell $A\subset \tilde W_l^0$ let $S_A$ be the set of canonical right cells $B$
such that $A\le_RB$ is not true. Then it follows from \cite[Main Theorem]{Ost1} that the full subcategory
$\T(\g,q)_A\subset \T(\g,q)$ consisting of direct sums of tilting modules
$T(\lambda), \lambda \in \Lambda^+\cap \left(\bigcup_{w\in S_A}C_w\right)$ is a thick tensor ideal
of $\T(\g,q)$ (i.e. $\T(\g,q)_A$ is closed under tensor products  with objects of $\T(\g,q)$). Moreover,
the ideal $\T(\g,q)_A$ admits a unique cover among thick tensor ideals: the ideal consisting of
direct sums of modules 
$T(\lambda), \lambda \in \Lambda^+\cap \left(\bigcup_{w\in \tilde S_A}C_w\right)$ where
$\tilde S_A=S_A\cup \{A\}$. Notice that the indecomposable objects $T(\lambda)$ appearing in this cover
and not appearing in $\T(\g,q)_A$ are labeled by the set 
$$P_A=\{ \lambda \in \Lambda^+\: |\: \lambda \in C_w\;
\mbox{for some}\; w\in A\; \mbox{and}\: \lambda \not \in C_{w'}\; \mbox{for any}\; w'\in B, B\in S_A\}.$$
It is clear that the set $P_A$ is finite if and only if $A$ is finite if and only if $A=A_e$ for some
distinguished nilpotent element $e$. It follows that for a distinguished nilpotent element $e$,
the tensor ideal $\T(\g,q)_{A_e}$ satisfies the assumptions of \cite[Theorem 2.4.1(2)]{CEO}.  

Now let $\cI_e$ be the tensor ideal of $\T(\g,q)$ which is maximal with respect to the following property:
the only identity morphisms contained in this ideal are $\Id_T$ where $T\in \T(\g,q)_{A_e}$ (such maximal ideal
exists by \cite[2.3.1]{CEO}). Using \cite[Theorem 2.4.1(2)]{CEO} we get the following

\begin{corollary}\label{abenv} {\em (}\cite[Theorem 5.4.1]{CEO}\em{)} 
For a distinguished nilpotent element $e$ the quotient category $\T(\g,q)/\cI_e$
admits an abelian monoidal envelope. 
\end{corollary}

\begin{definition} The abelian monoidal envelope from Corollary \ref{abenv} will be denoted by
$\C(\g,e,l,q)$. 
\end{definition}

It follows from the proof of \cite[Theorem 2.4.1(2)]{CEO} (which uses construction from \cite{BEO}) 
that the category $\C(\g,e,l,q)$ is finite; moreover its projective objects are $T(\lambda), \lambda \in P_A$.

\begin{example} (1) Assume $e=0$ (this element is never distinguished). Then $\T(\g,q)_{A_e}=0$
and $\cI_e=0$.

(2) Assume $e$ is the regular nilpotent element (which is always distinguished). 
Then the cell $A_e$ consists of only the identity element of $\tilde W_l$. It follows that $\cI_e$ is the maximal
proper tensor ideal of $\T(\g,q)$ and  $\C(\g,e,l,q)=\C(\g,l,q)$. All objects of this category
are projective, that is the category $\C(\g,l,q)$ is semisimple.
\end{example}

\begin{remark} In \cite{Ost1} it is assumed that $l$ is odd and undivisible; however the same proof works
if the interior of the alcove $C_l(\g)$ contains an integral weight. 
\end{remark}

\section{The subregular cell}
\subsection{Toy example}\label{toy}
Let $V$ be a two dimensional vector space over $\mathbb C$ and let $\Gamma \subset
SL(V)$ be a subgroup of even order. Let $\epsilon \in \Gamma$ be a unique element of order 2 ($\epsilon$ exists
since the order of $\Gamma$ is even; it is unique since $SL(V)$ contains a unique involution, negative identity).
We consider $V$ as an odd vector space and as an additive group of this vector space. Let $\Gamma \ltimes V$
be the super-group which is a semi-direct product of $V$ with $\Gamma$; we will consider $\epsilon$ as
an element of $\Gamma \ltimes V$. 

We consider the category $\Rep(\Gamma \ltimes V, \epsilon)$ of
finite dimensional representations of $\Gamma \ltimes V$ where $\epsilon$ acts as the parity automorphism,
see e.g. \cite[9.11]{EGNO}. Thus, $\Rep(\Gamma \ltimes V, \epsilon)$ is a symmetric finite tensor category.
As an abelian category it is equivalent to the  category of representations of algebra
$\wedge(V)$ in the category $\Rep(\Gamma)$ (or, equivalently, representations of $\wedge(V)$ with an action
of $\Gamma$  compatible in an obvious sense). The simple objects of $\Rep(\Gamma \ltimes V, \epsilon)$ are irreducible representations of $\Gamma$ where $V\subset \wedge(V)$ acts by zero;
let $\Irr(\Gamma)$ be the set of isomorphism classes of such objects.
The projective cover of the unit object is $\wedge(V)$; and the projective cover of $V_i\in \Irr(\Gamma)$
is $V_i\otimes \wedge(V)$. In particular, the composition factors of the projective cover of $V_i$ are
$V_i$ appearing twice and the irreducible summands of $V_i\otimes V$. Thus, the Cartan matrix of
the category $\Rep(\Gamma \ltimes V, \epsilon)$ is $2\Id+A_\Gamma$ where $A_\Gamma$ is the adjacency
matrix of the McKay graph of $\Gamma$ (recall that the vertices of the McKay graph are elements 
of $\Irr(\Gamma)$ and the number of edges between $V_i$ and $V_j$ is the multiplicity of $V_j$ as 
a direct summand of $V_i\otimes V$). 

\begin{proposition}\label{3.1}
(1) The category $\Rep(\Gamma \ltimes V, \epsilon)$ is of tame representation type.

(2) The algebra $\Ext_{\Rep(\Gamma \ltimes V, \epsilon)}^\bullet (\be,\be)$ is concentrated in even degrees; we have an algebra isomorphism $\Ext_{\Rep(\Gamma \ltimes V, \epsilon)}^{2\bullet} (\be,\be)\simeq S^\bullet(V)^\Gamma$.
\end{proposition} 

\begin{proof} (1) As an abelian category $\Rep(\Gamma \ltimes V, \epsilon)=\Rep(\Gamma \ltimes V)=$
$\Gamma-$equivariantization of representations of $\wedge(V)$. The result follows since it is well
known that the algebra $\wedge(V)$ is of tame representation type, see e.g. \cite{Ri}.

(2) The Koszul resolution for $\wedge(V)$ is equivariant with respect to the action of $GL(V)$. The result
follows.
\end{proof}

Recall that the McKay graph of $\Gamma$ is an affine Dynkin diagram; in particular it is a tree unless
$\Gamma$ is cyclic. Let $X$ be an arbitrary finite tree. We consider the following quiver $\tilde X$ with relations
associated to $X$: the set of vertices of $\tilde X$ is the same as the set of vertices of $X$. For every vertex $i\in X$ we will have loop $\delta_i$ at the vertex $i$ satisfying $\delta_i^2=0$. For every edge $i-j$ of $X$
we will have two arrows $i\to j$ and $j\to i$ in $\tilde X$; the composition of $i\to j$ and 
$j\to i$ is $\delta_i$ and the composition of $i\to j$ and $j\to k$ is zero if $k\ne i$. 
Similarly, the compositions of $i\to j$ with $\delta_i$ and $\delta_j$ are zero. Let $\Rep(\tilde X)$
be the category of finite dimensional representations of the quiver $\tilde X$. It is easy to see
that the Cartan matrix of $\Rep(\tilde X)$ is $2\Id +A(X)$ where $A(X)$ is the adjacency matrix of $X$. 

\begin{proposition}\label{tree}
Assume that the Cartan matrix of a block of some finite tensor category (over arbitrary 
algebraically closed field) is $2\Id+A$ where $A=A(X)$ is the adjacency matrix of a tree $X$. Then
this block is equivalent to $\Rep(\tilde X)$ as an abelian category.
\end{proposition}

\begin{proof} The indecomposable projective objects of the block in question are labeled by the vertices
of $X$. For any indecomposable projective $P_i$ its endomorphism algebra is 2-dimensional and local,
hence isomorphic to the algebra of dual numbers. Also for non-isomorphic indecomposable projectives $P_i$ 
and $P_j$, the space $\Hom(P_i,P_j)$ is one dimensional if there is an edge $i-j$ and is zero otherwise.
It follows that the algebra of endomorphisms $\End(P)$ of the generator $P=\oplus_iP_i$
(each $P_i$ appears with multiplicity 1) can be described via quiver with relations which has the same set of
vertices as $X$, for any vertex we have loop $\delta_i$ and for any edge $i-j$ we have arrows $i\to j, j\to i$.
The following relations are obvious: $\delta_i^2=0$, compositions of $i\to j$ with $\delta_i$ and $\delta_j$
are zero, and the composition of $i\to j$ and $j\to k$ is zero for $k\ne i$. Also the composition $i\to j$ and
$j\to i$ should be $\lambda_{ij}\delta_i$ for some scalars $\lambda_{ij}$. 

Recall that the functor
$\Hom(P,?)$ gives an equivalence of the block and of the category of finite dimensional right modules
over $\End(P)$ which is the same as representations of the quiver above. This equivalence send projective object $P_i$ to $\Hom(P,P_i)$, which can be described as all paths starting at the vertex $i$.. It is easy to see that
in the case $\lambda_{ij}=0$, the quiver representation $\Hom(P,P_i)$ has socle of length at least 2 (it is spanned
by $i\to j$ and by $\delta_i$). This is impossible for a block in finite tensor category: any indecomposable projective object
in such category has a simple socle, see \cite[Remark 6.1.5]{EGNO}. Thus we proved that $\lambda_{ij}\ne 0$
for any edge $i-j$. 

We claim now that we can rescale elements $i\to j$ and $\delta_i$ in such a way that $\lambda_{ij}=1$ for
all edges $i-j$. We use induction on the number of vertices of $X$. In the base case of one vertex there is
nothing to prove. Otherwise assume that $i$ is a leaf connected with only one vertex $j$. Then we can assume
that claim is true for $X\setminus \{i\}$. Now rescale $i\to j$ and $j\to i$ so their composition in one direction is $\delta_j$
and declare the other composition to be $\delta_i$. This completes the proof.
\end{proof}

\subsection{Tilting modules for subregular cell} The subregular cell was described by Lusztig \cite[3.7]{Luex}
for any Coxeter group. Namely, this cell consists of all elements with a unique reduced expression except for
the identity. In particular, the subregular canonical right cell $A_{sr}$ consists of all elements with a unique reduced
expression and starting from $s_0$ (affine reflection). This set is explicitly described in \cite[3.13]{Luex}, see
also \cite[Table 1]{Rasm}. Let $X_{sr}$ be the graph where the vertices are elements of $A_{sr}$; two
vertices are connected by an edge when the corresponding elements of $A_{sr}$ differ by multiplication
by a simple reflection on the right. Using \cite[3.13]{Luex} or \cite[Table 1]{Rasm} one finds that in the case when $A_{sr}$ is finite the graph $X_{sr}$ is a simply laced affine Dynkin diagram which is given in the table from Introduction.

The algorithm for computing characters of (quantized) tilting modules was proposed by Soergel \cite{So1};
this algorithm was proved to be correct in \cite{So2} (the proof in \cite{So2} was partially conditional since
it required knowledge of Kazhdan-Lusztig conjecture for affine Lie algebras at positive fractional level;
this conjecture was proved in \cite{KaTa}). Applying Soergel's algorithm to elements of $A_{sr}$ we get
the following (see \cite{Rasm} for closely related computations):

\begin{proposition}\label{nog2} 
(1) Assume that $\g$ is of type $D_n (n\ge 4), E_n (n=6,7,8)$ or $F_4$ (in particular,
$A_{sr}$ is finite). Then for any
$w\in A_{sr}$ the tilting module $T(w\cdot 0)$ has Weyl filtration of length 2; moreover
\begin{equation}\label{good char}
    \mbox{ch}(T(w\cdot 0))=\chi_{w\cdot 0}+\chi_{ws\cdot 0}
\end{equation}
where $s=s_w$ is a unique simple reflection such that $\ell(ws)=\ell(w)-1$.

(2) Assume that $A_{sr}$ is finite and assume that $\g$ is of type $B_n, C_n, G_2$.
Then \eqref{good char} holds for all elements $w\in A_{sr}$ except for exactly one $w_b\in A_{sr}$.
\end{proposition}  

Combining Proposition \ref{nog2} with \ref{qgtitling} (5) we get

\begin{corollary}\label{nog22}
Assume that $A_{sr}$ is finite and assume that $\g$ is not of type $G_2$, $B_n$, or $C_n$. 
Let $w,w'\in A_{sr}$. Then 
$$\dim \Hom(T(w\cdot 0), T( w'\cdot 0))=\begin{cases} 2 \;\; \mbox{   if } w=w',\\ 
1 \;\; \mbox{ if } w\ne w',
\mbox{ there is an edge }\; w-w'\; \mbox{in}\; X_{sr},\\
0 \;\; \mbox{ if } w\ne w',
\mbox{ there is no edge}\; w-w'\; \mbox{in}\; X_{sr}.\end{cases}
$$
\end{corollary}

Using {\em translation functors} one extends Proposition \ref{nog2} and Corollary \ref{nog22} 
to tilting modules $T(w\cdot \lambda)$, where $\lambda$ is in the interior of the fundamental 
alcove $C_e$. Moreover, using \cite[7.3.2]{So1} we get 

\begin{corollary}\label{nog222}
Assume that $A_{sr}$ is finite and assume that $\g$ is not of type $G_2$, $B_n$, or $C_n$. Let $\lambda \in P_{A_{sr}}$
but $\lambda$ is not contained in an interior of an alcove (in other words, $\lambda$ is on the wall). Then $T(\lambda)$ is simple and for any $\lambda' \in P_{A_{sr}}$ 
$$\dim \Hom(T(\lambda), T(\lambda'))=\dim \Hom(T(\lambda'),T(\lambda))=\left\{ \begin{array}{ccc} 1&\mbox{if}&\lambda=\lambda'\\ 0&\mbox{if}&\lambda \ne \lambda'. \end{array}\right.
$$
\end{corollary}

Let us now consider the case of $G_2$, $B_n$, and $C_n$. Suppose $A_{sr}$ is finite. The Coxeter diagram for $\tilde W_l$ is shown in Figure \ref{fig:coxeterBn} for types $B_n$ and $C_n$, and in Figure \ref{fig:coxeterG2} for type $G_2$. The graphs $X_{sr}$ for types $B_n$ (and $C_n$) and $G_2$ are depicted in Figures \ref{fig:Xsr_Bn} and \ref{fig:Xsr_G2} respectively.

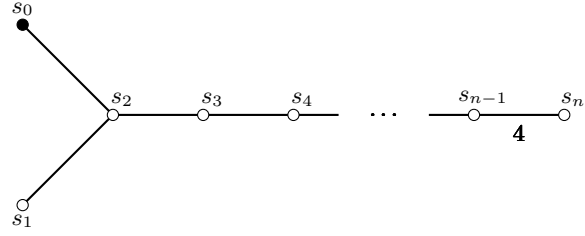
\begin{figure}[htbp]
  \centering
  \begin{tikzpicture}[scale=1.2, every node/.style={font=\small}]

\draw[black, thick] (1,0) -- (3.5,0);   
\draw[black, thick] (4.5,0) -- (6,0);   
\draw[black, thick] (1,0) -- (0,1);     
\draw[black, thick] (1,0) -- (0,-1);    

\foreach \x/\y/\name/\fillcolor in {0/-1/s_1/white, 1/0/s_2/white, 2/0/s_3/white, 3/0/s_4/white, 5/0/s_{n-1}/white, 6/0/s_n/white, 0/1/s_0/black} {
    \node[draw, circle, fill=\fillcolor, inner sep=1.5pt] at (\x,\y) {};
    \ifnum\x=0
        \ifnum\y=-1
            \node[below] at (\x,\y) {$s_1$};
        \else
            \node[above] at (\x,\y) {$s_0$};
        \fi
    \else
        \node[above] at (\x+0.1,\y) {$\name$};
    \fi

    \node at (4,0) {$\ldots$};
    \node at (5.5,-0.2) {4};
}

\end{tikzpicture}
  \caption{Coxeter diagram for $\tilde W_l$ for $\mathfrak g$ of type $B_n$ or $C_n$}
  \label{fig:coxeterBn}
\end{figure}

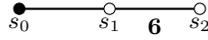
\begin{figure}[htbp]
  \centering
  \begin{tikzpicture}[scale=1.2, every node/.style={font=\small}]

\draw[black, thick] (0,0) -- (2,0);

\node[draw, circle, fill=black, inner sep=1.5pt] at (0,0) {};
\node[below] at (0,0) {$s_0$};

\node[draw, circle, fill=white, inner sep=1.5pt] at (1,0) {};
\node[below] at (1,0) {$s_1$};

\node[draw, circle, fill=white, inner sep=1.5pt] at (2,0) {};
\node[below] at (2,0) {$s_2$};

\node at (1.5,-0.2) {$\mathbf 6$};

\end{tikzpicture}
  \caption{Coxeter diagram for $\tilde W_l$ for $\mathfrak g$ of type $G_2$}
  \label{fig:coxeterG2}
\end{figure}

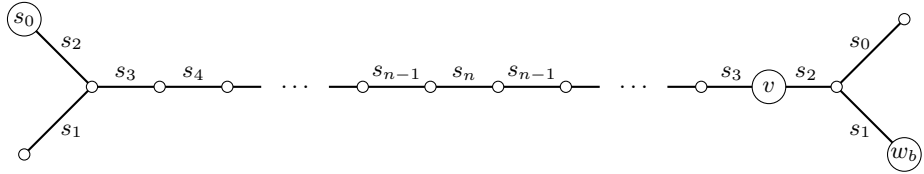
\begin{figure}[htbp]
  \centering
  \begin{tikzpicture}[scale=0.9, every node/.style={font=\small}]

\draw[black, thick] (1,0) -- (3.5,0);
\draw[black, thick] (4.5,0) -- (8.5,0);
\draw[black, thick] (9.5,0) -- (12,0);

\draw[black, thick] (0,1) -- (1,0);
\draw[black, thick] (0,-1) -- (1,0);

\draw[black, thick] (12,0) -- (13,1);
\draw[black, thick] (12,0) -- (13,-1);

\foreach \x/\y in {
    0/1, 0/-1,
    1/0, 2/0, 3/0,
    5/0, 6/0, 7/0, 8/0,
    10/0, 11/0, 12/0,
    13/1, 13/-1
}{
    \node[draw, circle, fill=white, inner sep=1.5pt] at (\x,\y) {};
}

\node at (0.7,0.65) {$s_2$};
\node at (0.7,-0.65) {$s_1$};

\node at (1.5,0.2) {$s_3$};
\node at (2.5,0.2) {$s_4$};

\node at (4,0) {$\ldots$};

\node at (5.5,0.2) {$s_{n-1}$};
\node at (6.5,0.2) {$s_n$};
\node at (7.5,0.2) {$s_{n-1}$};

\node at (9,0) {$\ldots$};

\node at (10.45,0.2) {$s_{3}$};
\node at (11.55,0.2) {$s_2$};

\node at (12.35,0.65) {$s_0$};
\node at (12.35,-0.65) {$s_1$};

\node[draw, circle, fill=white, inner sep=4.5pt] at (13,-1) {};
\node at (13,-1) {$w_b$};

\node[draw, circle, fill=white, inner sep=4.5pt] at (11,0) {};
\node at (11,0) {$v$};

\node[draw, circle, fill=white, inner sep=4.5pt] at (0,1) {};
\node at (0,1) {$s_0$};

\end{tikzpicture}
  \caption{Graph $X_{sr}$ for types $B_n$ and $C_n$}
  \label{fig:Xsr_Bn}
\end{figure}

\begin{figure}[htbp]
  \centering
  \begin{tikzpicture}[scale=1.2, every node/.style={font=\small}]

\draw[black, thick] (0,0) -- (6,0);
\draw[black, thick] (3,0) -- (3,-1);

\foreach \x/\y in {
    1/0,
    2/0,
    3/0,
    3/-1,
    4/0,
    5/0
}{
    \node[draw, circle, fill=white, inner sep=1.5pt] at (\x,\y) {};
}

\node[draw, circle, fill=white, inner sep=4.5pt] at (0,0) {};
\node at (0,0) {$s_0$};

\node at (0.5,0.15) {$s_1$};
\node at (1.5,0.15) {$s_2$};
\node at (2.5,0.15) {$s_1$};
\node at (3.45,0.15) {$s_2$};
\node at (4.55,0.15) {$s_1$};
\node at (5.5,0.15) {$s_0$};

\node at (3.2,-0.55) {$s_0$};

\node[draw, circle, fill=white, inner sep=4.5pt] at (6,0) {};

\node at (6,0) {$w_b$};

\node[draw, circle, fill=white, inner sep=4.5pt] at (4,0) {};

\node at (4,0) {$v$};

\end{tikzpicture}
  \caption{Graph $X_{sr}$ for type $G_2$}
  \label{fig:Xsr_G2}
\end{figure}
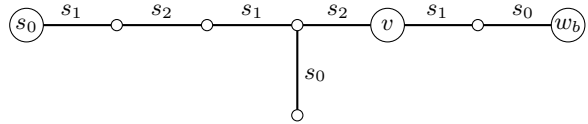

Let $w_b$ be the element of $\tilde W_l$ as in Proposition \ref{nog2} then
$$
\begin{cases}
    w_b = s_0s_2s_3\ldots s_{n-1}s_n s_{n-1}\ldots s_3s_2 s_1, \text{ if } \mathfrak{g} \text{ is of type } B_n \text{ or } C_n,\\
    w_b = s_0s_1s_2s_1s_2s_1s_0, \text{ if } \mathfrak{g} \text{ is of type } G_2.
\end{cases}
$$
We have
\begin{equation}\label{bad char}
    \mbox{ch}(T(w_b\cdot 0))=\chi_{w_b\cdot 0}+\chi_{vs\cdot 0}+\chi_{vt\cdot 0} + \chi_{v\cdot 0},
   \end{equation}
where $s = s_1, t=s_2$ in types $B_n$ and $C_n$,  $s=s_0, t =s_1$ in type  $G_2$, and $v=w_bst$.

Note that elements $w_b, v, vt$ are all in $A_{sr}$, whereas $vs$ is in some cell $B$ with $B\in S_{A_{sr}}$.

Another computation shows that 
$$
\mbox{ch}(T(v\cdot 0))=\chi_{v\cdot 0}+\chi_{vr\cdot 0},
$$
$$
\mbox{ch}(T(vt\cdot 0))=\chi_{vt\cdot 0}+\chi_{v\cdot 0},
$$
$$
\mbox{ch}(T(w_bt\cdot 0))= \chi_{w_bt\cdot 0}+\chi_{w_b\cdot 0} + \chi_{vst\cdot 0}+\chi_{vs\cdot 0}+\chi_{vt\cdot 0}+\chi_{v\cdot 0}.
$$

Our goal is to understand $\Hom(T(\lambda), T(\lambda'))$ modulo $\mathcal I_{e_{sr}}$. Let us denote the quotient by $\Hom_\C(T(\lambda), T(\lambda'))$ (here $\C$ stands for $\C(\mathfrak g, e_{sr}, l,q)$).

For each simple reflection $s_i$, let $\varTheta_{s_i}$ denote the wall-crossing translation functor corresponding to $s_i$. Then $T(v\cdot 0)$ is a direct summand in $\varTheta_r T(vr\cdot 0)$, $T(vt \cdot 0)$ is a direct summand in $\varTheta_t T(v\cdot 0)$, and $T(w_b\cdot 0)$ is a direct summand in $\varTheta_s T(vt\cdot 0)$. At the same time we have
$$
\varTheta_r T(w_b\cdot 0)=T(vs\cdot 0),
$$
$$
\varTheta_t T(w_b\cdot 0) = T(w_bt\cdot 0)\oplus T(vt\cdot 0),
$$
$$
\varTheta_s T(w_b\cdot 0)= T(w_b\cdot 0)\oplus T(w_b\cdot 0).
$$

\begin{proposition} \label{prop:bad weight homs} $ $\\ \vspace{2mm}
 (1) $\dim \Hom_\C(T(v\cdot 0), T(w_b\cdot 0))\le \dim \Hom_\C(T(vr\cdot 0), \varTheta_r T(w_b\cdot 0))=0.$
\\
\vspace{2mm}(2)  $\dim \Hom_\C(T(vt\cdot 0), T(w_b\cdot 0)) \le \dim\Hom_\C(T(v\cdot 0), \varTheta_t T(w_b\cdot 0)) \le 1.$\\
(3) $\dim \Hom_\C(T(w_b\cdot 0), T(w_b\cdot 0)) \le \dim\Hom_\C(T(vt\cdot 0), \varTheta_s T(w_b\cdot 0))\le 2.$

\end{proposition}
\begin{proof}
(1) This is a corollary of the computation above and the observation that $vs$ lies outside the subregular cell, so morphisms into $T(vs\cdot 0)$ lie in $\mathcal I_{e_{sr}}$.

(2) Note that the element $w_bt$ is outside of the subregular cell. We have $$\dim\Hom(T(v\cdot 0), \varTheta_t T(w_b\cdot 0)))=2,$$ however, one of the two homomorphisms factors through the summand $T(w_bt\cdot 0)$ of $\varTheta_t T(w_b\cdot 0)$, and thus lies in $\mathcal I_{e_{sr}}$.

(3) We have 
$$
\dim\Hom_\C(T(vt\cdot 0), \varTheta_sT(w_b\cdot 0))=2\dim\Hom_\C(T(vt\cdot 0), T(w_b\cdot 0))\le 2\cdot 1
$$
by part (2).
\end{proof}

Applying the result above together with Proposition \ref{nog2} (2), we get a result analogous to Corollary \ref{nog22} for the quotient $\Hom_\C(T(\lambda), T(\lambda'))$.

\begin{corollary} \label{cor:bad weight homs}
    Assume $A_{sr}$ is finite and $\mathfrak g$ is of type $B_n, C_n$, or $G_2$. Let $w, w'\in A_{sr}$. Then
    $$
    \dim\Hom_\C(T(w\cdot 0), T(w'\cdot 0))\le \begin{cases}
        2 \;\; \text{ if } w=w',\\
        1 \;\; \text{ if } w\neq w', \text{ there is an edge } w-w' \text{ in } X_{sr},\\
        0 \;\; \text{ if } w\neq w', \text{ there is no edge } w-w' \text{ in } X_{sr}.
    \end{cases}
    $$
\end{corollary}

Using translation functors, we can extend this to tilting modules $T(w\cdot \lambda)$ for all $\lambda$ in the interior of $C_e$. As before, using translation onto the walls, we get the analog of Corollary \ref{nog222} in this case.
\begin{corollary}
    \label{cor:on the  wall}
Assume that $A_{sr}$ is finite and assume that $\g$ is of type $G_2$, $B_n$, or $C_n$. Let $\lambda \in P_{A_{sr}}$
but $\lambda$ is not contained in an interior of an alcove (in other words, $\lambda$ is on the wall). Then the image of $T(\lambda)$ in $\mathcal C(\mathfrak g, e_{sr}, l, q)$ is simple and for any $\lambda' \in P_{A_{sr}}$ 
$$\dim \Hom_\C(T(\lambda), T(\lambda'))=\dim \Hom_\C(T(\lambda'),T(\lambda))=\left\{ \begin{array}{ccc} 1&\mbox{if}&\lambda=\lambda'\\ 0&\mbox{if}&\lambda \ne \lambda'. \end{array}\right.
$$
\end{corollary}

The unit object of the category $\T(\g,q)$ is $\be=T(0)$. Using the knowledge of the characters
of tilting modules from Proposition \ref{nog2} and  the computation for $\mbox{ch}(T(w_b\cdot 0))$,  we obtain

\begin{corollary}\label{nog2unit}
Assume that $A_{sr}$ is finite (and $\mathfrak g$ is of any type that permits it). Let $\lambda \in P_{A_{sr}}$.
Then 
$$\dim \Hom(\be, T(\lambda))=\dim \Hom(T(\lambda),\be)=\left\{ \begin{array}{ccc} 1&\mbox{if}&\lambda=s_0\cdot 0\\ 0&\mbox{if}&\lambda \ne s_0\cdot 0 \end{array}\right.
$$
where $s_0$ is the affine simple reflection of $\tilde W_l$.
\end{corollary}
 
\subsection{Principal block of $\C(\g,e_{sr},l,q)$}
Recall (see Section \ref{t ideal}) the set of weights $P_A$ attached to a canonical cell $A$.
The indecomposable projective objects of the category $\C(\g,e_{sr},l,q)$ are $T(\lambda)$ where
$\lambda \in P_{A_{sr}}$. The linkage principle (see \ref{qgtitling} (6)) immediately implies that
the category $\C(\g,e_{sr},l,q)$ decomposes into summands which correspond to intersections
of $\tilde W_l-$orbits with the set $P_{A_{sr}}$ (these summands can be decomposed even more).

\begin{lemma}\label{projsr} 
(1) The projective cover of the unit object in $\C(\g,e_{sr},l,q)$ is
$T(s_0\cdot 0)$ (recall that $s_0$ is the affine simple reflection in $\tilde W_l$). 

(2) The category $\C(\g,e_{sr},l,q)$ is unimodular (see \cite[Definition 6.5.7]{EGNO}).

(3) The projective objects of the principal block of $\C(\g,e_{sr},l,q)$ are among \\
$T(w\cdot 0), w\in A_{sr}$.
\end{lemma}

\begin{proof} (1) is clear from Corollary \ref{nog2unit}; we also see that the socle of the projective cover of $\be$ is $\be$. This means that the distinguished invertible object (see
\cite[Definition 6.4.4]{EGNO}) of the category $\C(\g,e_{sr},l,q)$ is $\be$, that is this
category is unimodular and we get (2). Finally, (3) is a consequence of (1) and the 
linkage principle \ref{qgtitling} (6).
\end{proof}

\begin{lemma}\label{projubf}
    Let $\C$ be a unimodular braided finite tensor category. Then

    (1) The head and the socle of any indecomposable projective in $\C$
are isomorphic.

(2) For projective objects $P,P'\in \C$ we have
$$\dim \Hom(P,P')=\dim \Hom(P',P).$$

(3) Let $P$ be an indecomposable projective object such that $\dim \Hom(P,P)=1$.
Then $P$ is simple and it does not appear as a subquotient of any other indecomposable
projective.
\end{lemma}

\begin{proof}
    For any braided finite tensor category we have $X\simeq X^{**}\simeq \phantom{}^{**}X$ 
    for any object $X$, see \cite[Proposition 8.10.6]{EGNO}. Also in any finite tensor category
    the socle of the projective cover of a simple object $X$ is $X^{**}\otimes X_\rho^*$
    where $X_\rho^*$ is the socle of the projective cover of $\be$, see 
\cite[6.4]{EGNO}. This proves (1).

For (2) we have $\Hom(P,P')=\Hom(\be, P'\otimes P^*)$ and $$\Hom(P',P)=\Hom(\be, P\otimes (P')^*)=
\Hom(\be, (P'\otimes \phantom{}^*P)^*)=\Hom(\be, (P'\otimes P^*)^*)$$ where we use isomorphism $\phantom{}^*P\simeq (\phantom{}^*P)^{**}\simeq P^*$ from \cite[Proposition 8.10.6]{EGNO}.
Thus both $\dim \Hom(P,P')$ and $\dim \Hom(P',P)$ equal the number of projective covers of $\be$
appearing as a direct summand of projective object $P'\otimes P^*$.

Let $P$ be as in (3). Then it follows from (1) that the head and the socle of $P$ coincide,
so $P$ is simple. Since $P$ is both projective and injective it can't be a subquotient
of an indecomposable module except for itself, so (3) follows.
\end{proof}

\begin{corollary} \label{cor:self-dual}
    Let $\mathfrak g$ be of type $B_n, C_n, F_4, G_2, D_{2n}, E_7$, or $E_8$.
    Then all simple objects of $\C=\C(\mathfrak g, e_{sr}, l, q)$ are self-dual. 
\end{corollary}
\begin{proof}
    It is well-known that in this case all simple representations of $\mathfrak g$ are self-dual. It follows from properties (1) and (2) in Section \ref{qgtitling} that the tilting modules $T(\lambda)$ are also self-dual. 

    If $L$ is a simple object in $\C$ and $P$ is its projective cover, then $P=P^*=T(\lambda)$ for some $\lambda$, and so $P$ is the injective hull of $L^*$. Lemma \ref{projubf} then implies that $L$ is the socle of $P$, and so $L= L^*$. 
\end{proof}

\begin{proposition}\label{Cartansr}
   The projective objects of the principal block of the category $\C(\g,e_{sr},l,q)$ are 
   $T(w\cdot 0), w\in A_{sr}$. The Cartan matrix of the principal block is
   $2\Id+A(X_{sr})$ where $A(X_{sr})$ is the adjacency matrix of the graph $X_{sr}$.
\end{proposition}

\begin{proof}
By Lemma \ref{projsr} (3) the projective objects of the the principal block of the category $\C(\g,e_{sr},l,q)$ are some of $T(w\cdot 0), w\in A_{sr}$. The entries of the Cartan
matrix of this block are dimensions of $\Hom(T(w\cdot 0), T(w'\cdot 0))$ modulo the ideal $\cI_{e_{sr}}$. By Corollaries \ref{nog22} and \ref{cor:bad weight homs}  the diagonal entries are $\le 2$ and the off-diagonal entries are $\le 1$. By  Lemma \ref{projubf} (3) we have no diagonal entries equal to 1, and by 
Lemma \ref{projubf} (2) the Cartan matrix is symmetric. It follows that the Cartan matrix
is $2\Id+A(\tilde X_{sr})$ where $A(\tilde X_{sr})$ is the adjacency matrix of some subgraph
$\tilde X_{sr}$ of $X_{sr}$. Using the same argument as in the proof of \cite[Theorem 6.6.1]{EGNO}
one shows that the Cartan matrix of the principal block must be degenerate. It is well known
that $2\Id+A(\tilde X_{sr})$ is non-degenerate for any proper subgraph of the affine Dynkin
diagram $X_{sr}$. Hence $\tilde X_{sr}=X_{sr}$. 
\end{proof}

\subsection{Proof of Theorem \ref{main blocks}} By Proposition \ref{Cartansr} the Cartan
matrix of the principal block of $\C(\g,e_{sr},l,q)$ is $2\Id+A(X_{sr})$ where $X_{sr}$ is
some affine Dynkin diagram not of type $A$, so it is a tree. Thus, by Proposition \ref{tree}
(applied to this block and to the category $\Rep(\Gamma \ltimes V)$)
we see that the principal block is equivalent to $\Rep(\Gamma \ltimes V)$ where 
$V$ is as in \ref{toy} and $\Gamma \subset SL(V)$ is a finite subgroup with McKay graph $X_{sr}$.
Using the translation functors (which descend to the category $\T(\g,q)/\cI_{e_{sr}}$) we see
that all the other blocks of $\C(\g,e_{sr},l,q)$ involving $T(w\cdot \lambda)$ where $\lambda$
is in the interior of the fundamental alcove are equivalent to the principal block, and hence
to the category $\Rep(\Gamma \ltimes V)$. The other blocks involving $T(\lambda)$ with $\lambda$
on the wall are trivial by Corollaries \ref{nog222} and \ref{cor:on the  wall}. The remaining statements of 
Theorem \ref{main blocks} follow from Proposition \ref{3.1}.

\begin{remark}
    Let $\tilde \cI_{e_{sr}}$ be the tensor ideal of $\T(\g,q)$ generated by
    $\Id_T$ where $T\in \T(\g,q)_{A_{e_{sr}}}$ (see Section \ref{t ideal}). The calculations
    in the proof of Theorem \ref{main blocks} show that the category $\T(\g,q)/\tilde \cI_{e_{sr}}$ contains exactly two indecomposable objects $T$ with $\Hom(\be,T)\ne 0$, namely $T(0)=\be$ and $T(s_0\cdot 0)$ and$\Hom(\be,T)$ is one dimensional in both of these cases. It follows that the functor
    $\Hom(\be,?)$ has exactly one proper non-trivial subfunctor generated
    by space $\Hom(\be,T(s_0\cdot 0))$. By a theorem of Coulembier 
    \cite[Theorem 3.1.1]{Co} this implies that the category $\T(\g,q)/\tilde \cI_{e_{sr}}$ has exactly one proper non-trivial tensor ideal which must
    coincide with the ideal of negligible morphisms. We deduce that $\cI_{e_{sr}}=\tilde \cI_{e_{sr}}$ for all cases when $e_{sr}$ is distinguished
    nilpotent element.
\end{remark}

\section{Category $\C(G_2,G_2(a_1),7,q)$}
Let $l = 7$, let $q$ be a root of unity of order $7$ or $14$, and let $\mathfrak g$ be of type $G_2$. Recall that the  Coxeter number $h=\langle \rho, \beta_0\rangle+1$ for $\mathfrak g$ is equal to $6$.  It is easy to see that in this case, the fundamental alcove $C_l(\mathfrak g)$ contains exactly one integral dominant weight (zero) in its interior. 

Let us number the vertices of $X_{sr}$ as shown in Figure \ref{fig:Xsr_G2_labeled}. Recall that each vertex of $X_{sr}$ corresponds to an element $w\in A_{sr}$. For each $i=0,1,\ldots, 7$, if $w$ is the element corresponding to vertex $i$, we will refer to the alcove $C_w$ as $C_i$.

\begin{figure}[htbp]
  \centering
  \begin{tikzpicture}[scale=1.2, every node/.style={font=\small}]

\draw[black, thick] (0,0) -- (6,0);
\draw[black, thick] (3,0) -- (3,-1);

\foreach \x/\y in {
    1/0,
    2/0,
    3/0,
    3/-1,
    4/0,
    5/0
}{
    \node[draw, circle, fill=white, inner sep=1.5pt] at (\x,\y) {};
}

\node[draw, circle, fill=white, inner sep=4.5pt] at (0,0) {};
\node at (0,0) {$0$};
\node[draw, circle, fill=white, inner sep=4.5pt] at (1,0) {};
\node at (1,0) {$1$};
\node[draw, circle, fill=white, inner sep=4.5pt] at (2,0) {};
\node at (2,0) {$2$};
\node[draw, circle, fill=white, inner sep=4.5pt] at (3,0) {};
\node at (3,0) {$3$};
\node[draw, circle, fill=white, inner sep=4.5pt] at (3,-1) {};
\node at (3,-1) {$4$};
\node[draw, circle, fill=white, inner sep=4.5pt] at (4,0) {};
\node at (4,0) {$5$};
\node[draw, circle, fill=white, inner sep=4.5pt] at (5,0) {};
\node at (5,0) {$6$};

\node at (0.5,0.2) {$s_1$};
\node at (1.5,0.2) {$s_2$};
\node at (2.5,0.2) {$s_1$};
\node at (3.5,0.2) {$s_2$};
\node at (4.5,0.2) {$s_1$};
\node at (5.5,0.2) {$s_0$};

\node at (3.2,-0.5) {$s_0$};

\node[draw, circle, fill=white, inner sep=4.5pt] at (6,0) {};

\node at (6,0) {$7$};

\end{tikzpicture}
  \caption{Graph $X_{sr}$ for type $G_2$ (labeled)}
  \label{fig:Xsr_G2_labeled}
\end{figure}
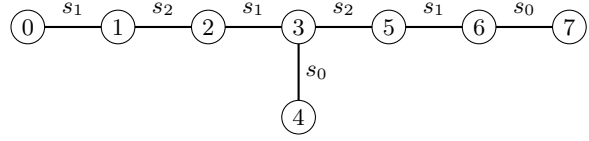

The union of alcoves $C_w$ for $w$ in $A_{sr}$ is shown in Figure \ref{fig:subreg_cell_G2} in red.

\begin{figure}[htbp]
  \centering
  \begin{tikzpicture}[scale=2.5]

\fill[red!20] (0.866,0) -- (0.866,0.5) -- (2.598,1.5) -- (1.732,0) -- cycle;

\fill[red!20] (1.732,0) -- (1.732,1) -- (3.464,0) -- cycle;

\draw[black] (0,0) -- (0.866,0); 
\draw[black] (0.866,0) -- (1.732,0); 
\draw[black] (1.732,0) -- (3.464,0); 
\draw[black] (0,0) -- (0.866,0.5); 
\draw[black] (0.866,0.5) -- (2.598,1.5); 
\draw[black] (0.866,0) -- (0.866,0.5); 
\draw[black] (0.866,0.5) -- (1.732,0); 
\draw[black] (1.732,0) -- (1.299,0.75); 
\draw[black] (1.732,0) -- (1.732,1); 
\draw[black] (1.732,0) -- (2.598,1.5); 
\draw[black] (1.732,1) -- (3.464,0); 
\draw[black] (1.732,0) -- (2.598,0.5); 
\draw[black] (2.598,0) -- (2.598,0.5); 

\node at (1.15, 0.15) {$C_0$};
\node at (1.3, 0.4) {$C_1$};
\node at (1.58, 0.62) {$C_2$};
\node at (1.9, 0.62) {$C_3$};
\node at (2.1, 1) {$C_4$};
\node at (2.2, 0.45) {$C_5$};
\node at (2.35, 0.15) {$C_6$};
\node at (2.85, 0.15) {$C_7$};

\end{tikzpicture}
  \caption{Subregular cell for type $G_2$}
  \label{fig:subreg_cell_G2}
\end{figure}
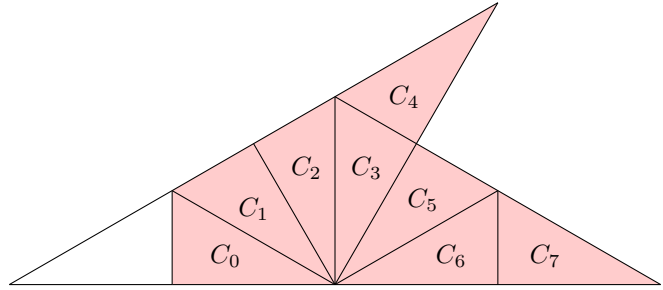

\begin{figure}[htbp]
  \centering
  \begin{tikzpicture}[scale=3]

\fill[green!20] (0.618,0.214) -- (3.088,0.214) -- (2.473,1.284) -- cycle;

\node[draw, circle, fill=white, inner sep=1pt] at (0.618,0.214) {};

\node[draw, circle, fill=black, inner sep=0.8pt] at (0.865, 0.214) {};
\node[draw, circle, fill=black, inner sep=0.8pt] at (0.989, 0.428) {};
\node[draw, circle, fill=red, inner sep=0.8pt] at (1.112, 0.214) {};
\node[draw, circle, fill=black, inner sep=0.8pt] at (1.359, 0.214) {};
\node[draw, circle, fill=black, inner sep=0.8pt] at (1.606, 0.214) {};
\node[draw, circle, fill=black, inner sep=0.8pt] at (1.853, 0.214) {};
\node[draw, circle, fill=black, inner sep=0.8pt] at (2.1, 0.214) {};
\node[draw, circle, fill=red, inner sep=0.8pt] at (2.347, 0.214) {};
\node[draw, circle, fill=black, inner sep=0.8pt] at (2.594, 0.214) {};
\node[draw, circle, fill=red, inner sep=0.8pt] at (2.841, 0.214) {};
\node[draw, circle, fill=red, inner sep=0.8pt] at (1.236, 0.428) {};
\node[draw, circle, fill=black, inner sep=0.8pt] at (1.483, 0.428) {};
\node[draw, circle, fill=black, inner sep=0.8pt] at (1.73, 0.428) {};
\node[draw, circle, fill=black, inner sep=0.8pt] at (1.977, 0.428) {};
\node[draw, circle, fill=red, inner sep=0.8pt] at (2.224, 0.428) {};
\node[draw, circle, fill=black, inner sep=0.8pt] at (2.471, 0.428) {};

\node[draw, circle, fill=black, inner sep=0.8pt] at (1.36, 0.642) {};
\node[draw, circle, fill=red, inner sep=0.8pt] at (1.607, 0.642) {};
\node[draw, circle, fill=red, inner sep=0.8pt] at (1.854, 0.642) {};
\node[draw, circle, fill=black, inner sep=0.8pt] at (2.101, 0.642) {};
\node[draw, circle, fill=black, inner sep=0.8pt] at (1.731, 0.856) {};
\node[draw, circle, fill=black, inner sep=0.8pt] at (1.978, 0.856) {};
\node[draw, circle, fill=red, inner sep=0.8pt] at (2.108, 1.07) {};

\node[draw, circle, fill=black, inner sep=0.8pt] at (0,0) {};
\node at (-0.1,-0.1) {$-\rho$};
\node at (0.518, 0.114) {$(0,0)$};

\draw[black] (0,0) -- (0.866,0); 
\draw[black] (0.866,0) -- (1.732,0); 
\draw[black] (1.732,0) -- (3.464,0); 
\draw[black] (0,0) -- (0.866,0.5); 
\draw[black] (0.866,0.5) -- (2.598,1.5); 
\draw[black] (0.866,0) -- (0.866,0.5); 
\draw[black] (0.866,0.5) -- (1.732,0); 
\draw[black] (1.732,0) -- (1.299,0.75); 
\draw[black] (1.732,0) -- (1.732,1); 
\draw[black] (1.732,0) -- (2.598,1.5); 
\draw[black] (1.732,1) -- (3.464,0); 
\draw[black] (1.732,0) -- (2.598,0.5); 
\draw[black] (2.598,0) -- (2.598,0.5); 

\node at (0.962, 0.17) {$\omega_1$};
\node at (1.06, 0.48) {$\omega_2$};

\end{tikzpicture}
  \caption{Integral dominant weights contained in $P_{A_{sr}}$}
  \label{fig:weights_in_cell}
\end{figure}
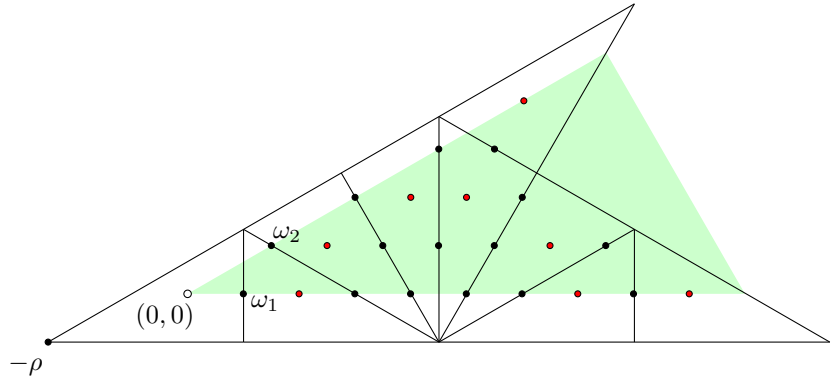

Recall that $P_{A_{sr}}$ is defined as the set of all integral dominant weights contained in the union of $C_0, C_1,\ldots, C_7$, that are not on the walls of any alcoves $C_{w'}$ with $w'\in B$, such that $B\in S_{A_{sr}}$.
A straightforward computations shows that $P_{A_{sr}}$ has $23$ elements, $8$ of which are of the form $w\cdot 0$ for $w\in A_{sr}$ and are contained in the interiors of the corresponding cells $C_w$, whereas the remaining $15$ weights are on the walls. See Figure \ref{fig:weights_in_cell}.

Let $\lambda_i, i=1,\ldots 8$ denote the unique integral dominant weight in the interior of $C_i$. Let $P_i$ denote $T(\lambda_i)$ considered as an object of $\C=\C(G_2,G_2(a_1), 7, q)$. Recall that by Proposition \ref{Cartansr} and Lemma \ref{projsr}, objects $P_0, P_1,\ldots, P_7$ are the projective objects of the principal block of $\C$, and $P_0$ is the projective cover of $\be$. Let $L_i$ denote the simple head of $P_i$, so that $L_0=\be$.

The Cartan matrix $A$ of the principal block (also computed in Proposition \ref{Cartansr}) has the following properties:
\begin{itemize}
    \item $A_{ij}=\dim\Hom_\C(P_i, P_j)=[P_j:L_i]$;
    \item $A_{ii}=2$;
    \item if $i\neq j$, $A_{ij}=1$ if there is an edge between vertices $i$ and $j$ in $X_{sr}$ (see Figure \ref{fig:Xsr_G2_labeled}), and $A_{ij}=0$ otherwise.
\end{itemize}

We obtain the following relations in the Grothendieck ring $[\C]$ of $\C$:
\begin{equation}
\label{eq:proj_and_simple}
\begin{aligned}
    [P_0]=2[L_0]+[L_1];\\
    [P_1]=[L_0]+2[L_1]+[L_2];\\
    [P_2]=[L_1]+2[L_2]+[L_3];\\
    [P_3]=[L_2]+2[L_3]+[L_4]+[L_5];\\
    [P_4]=[L_3]+2[L_4];\\
    [P_5]=[L_3]+2[L_5]+[L_6];\\
    [P_6]=[L_5]+2[L_6]+[L_7];\\
    [P_7]=[L_6]+2[L_7].
\end{aligned}    
\end{equation}

\subsection{Proof of Theorem \ref{main 7}} \label{{section: proof of theorem 1.3}}
Let us recall the statement of Theorem \ref{main 7}:
\begin{enumerate}
    \item  The category $\C(G_2,G_2(a_1),7,q)$ has 15 trivial blocks and one block of type $\tilde E_7$.
In particular it has 23 simple objects.

\item We have $\FP(\C(G_2,G_2(a_1),7,q))=294(7+15[3]_7+12[5]_7)\approx 18324.416384$.

\item The category $\C(G_2,G_2(a_1),7,q)$ has stable Chevalley property: tensor products of simple 
objects are direct sums of simples and projectives.

\item The M\"uger center of the category $\C(G_2,G_2(a_1),7,q)$ is equivalent to $\Rep(S_3)$
(where $S_3$ is the symmetric group on three letters).
\end{enumerate}

\begin{remark}
    Big part of the proof below relies on computations of tensor products of tilting modules $T(\lambda)$ for $\lambda\in P_{A_{sr}}$, and some auxiliary computations, such as matrices of tensor multiplication by generating objects $T(\omega_1)$ and $T(\omega_2)$ and their eigenvalues. These computations were made with a simple Python program. You can find the full code on \href{https://github.com/autiralova/A-non-semisimple-Witt-class-code}{GitHub} (see \cite{OU}).
\end{remark}

\begin{proof}
(1) The statement follows from the computation above (see Figure \ref{fig:weights_in_cell}), Proposition \ref{Cartansr}, and Corollary \ref{cor:on the  wall}.

(2) Let $\omega_1$ and $\omega_2$ be the two fundamental weights for $\mathfrak g$. Let us totally order the elements of $P_{A_{sr}}$ in some way consistent with the partial dominance order, so that $\omega_1$ is the first, and $\omega_2$ is the second element. Denote the $i$'th element by $\mu_i$.
        
        Using Soergel's algorithm (see \cite{So2}) one can explicitly compute the characters of all tilting modules $T(\lambda)$ with $\lambda$ not too large. It is then a straightforward computation to find the matrices $M(\omega_1)$ and $M(\omega_2)$ of tensor multiplication by $T(\omega_1)$ and $T(\omega_2)$ respectively. Moreover, as $T(\omega_1)$ and $T(\omega_2)$ tensor generate the whole category $\mathcal T(\mathfrak g, q)$, one can find the matrix $M(\mu)$ of tensor multiplication by $T(\mu)$ for all $\mu\in P_{A_{sr}}$ as follows.

         Matrix $M(\mu_k)$ is a square matrix of size $23$, and $M(\mu_k)_{ij}$ is the multiplicity $[T(\mu_k)\otimes T(\mu_j):T(\mu_i)]$ of $T(\mu_i)$ as a direct summand in $T(\mu_k)\otimes T(\mu_j)$ for all $i,j=1,\ldots, 23$. Define the augmented matrix $\widetilde M(\mu_k)$ by adding row and column number $0$ to $M(\mu_k)$ (so that $\widetilde M(\mu_k)$ is a 24-by-24 matrix). Put $M(\mu_k)_{0j}=0$ for all $j=0,1,\ldots, 23$, and $\widetilde M(\mu_k)_{i0} =\delta_{ik}$, i.e. $\widetilde M(\mu_k)_{i0}$ is the multiplicity  $[T(\mu_k)\otimes \be: T(\mu_i)]$. 

         In $\C=\C(G_2, G_2(a_1),7, q)$, we have 
         $$ T(\omega_1)^{\otimes a}\otimes T(\omega_2)^{\otimes b} = \bigoplus_{i=1}^{23}
           T(\mu_i)^{\oplus (\widetilde M(\omega_1)^a \cdot \widetilde M(\omega_2)^b)_{i0}}.
         $$
         Let $X=[T(\omega_1)], Y=[T(\omega_2)]$ be  the classes of the generating elements in the Grothendieck ring $[\C]$ of $\C$. For each pair of integers $(a,b)$, such that $\mu_i=a\omega_1+b\omega_2$, we obtain a linear equation in $[\C]$:
         $$
         \sum_{j=1}^{i} (\widetilde M(\omega_1)^a \cdot \widetilde M(\omega_2)^b)_{j0}\cdot [T(\mu_j)] = X^aY^b.
         $$
         The solution of this upper-triangular system of $23$ linear equations is a collection of polynomials $\{f_\mu~|~\mu\in P_{A_{sr}}\}$, such that
         $$
         f_\mu(X,Y) = [T(\mu)] \text{ in } [\C],
         $$
         and if $\mu=a\omega_1+b\omega_2$ then the leading term of $f_\mu(X,Y)$ is $X^aY^b$.

         By definition, $\FP T(\omega_i)$ is the largest nonnegative real eigenvalue of $M(\omega_i)$. We compute directly that 
         $$
         \FP (T(\omega_1)) = 1+2\cdot [3]_7,
         $$
         $$
         \FP (T(\omega_2)) = [2]_7+3\cdot [3]_7,
         $$
         for some choice of ordering of $\omega_1,\omega_2$.
         
         We get
         $$M(\mu)=f_\mu(M(\omega_1), M(\omega_2)),$$
         $$
         \FP(T(\mu)) = f_\mu(\FP (T(\omega_1)), \FP (T(\omega_2))).
         $$

         Let $L(\mu)$ denote the head of $T(\mu)$ in $\C$. If $\mu$ is on a wall of some alcove, by Corollary \ref{cor:on the  wall} we get $T(\mu)=L(\mu)$. If $\mu$ is in the interior of some alcove, i.e. $T(\mu) = P_j$ (for some $j$) is in the principal block, then $L(\mu)=L_j$. In this case $\FP(L(\mu))$ can be computed using the relations \ref{eq:proj_and_simple}.
         
         Then the Frobenius-Perron dimension of the category $\C$ is defined as
         $$
         \FP(\C) = \sum_{i=1}^{23} \FP(T(\mu_i))\cdot \FP(L(\mu_i)),
         $$
         and a direct computation shows that 
         $$
         \FP(\C) = 294(7+15[3]_7+12[5]_7).
         $$
         
(3) Let $\mu$ be on the wall, so that $T(\mu)=L(\mu)$. Then $T(\mu)$ is projective and the tensor product of it with any simple module is the direct sum of some indecomposable projectives.

    Thus, we only need to figure out the tensor products $L_i\otimes L_j$ for $1\le i,j \le 7$ (since $L_0=\be$).

    The matrices of multiplication of $\{T(\mu)~|~\mu\in P_{A_{sr}}\}$ by $L_i$ for $i-0,\ldots, 7$, can be found by using such matrices for $P_i$ and the relations \ref{eq:proj_and_simple}.

    Our computation shows the following:
    \begin{lemma} \label{lemma:tens_simp_by_proj}
        (a) $L_7$ is invertible of order two, $\FP(L_7) = 1$, 
        \begin{align*}
            P_7 = P_0\otimes L_7,\\
            P_6=P_1\otimes L_7,\\
            P_5 = P_2\otimes L_7,\\
            P_3 = P_3\otimes L_7,\\
            P_4 = P_4\otimes L_7.
        \end{align*}
        (b) $\FP(L_4) = 2,$ 
        \begin{align*}
            P_0\otimes L_4 = P_7\otimes L_4 = P_4,\\
            P_1\otimes L_4 = P_6\otimes L_4 = P_3,\\
            P_3\otimes L_4 = P_1\oplus P_3\oplus P_6,\\
            P_4\otimes L_4 = P_0\oplus P_4 \oplus P_7,\\
            P_2\otimes L_4 = P_5\otimes L_4 = P_2\oplus P_5.
        \end{align*}
    \end{lemma}
    Using relations \ref{eq:proj_and_simple} and Lemma \ref{lemma:tens_simp_by_proj}, we get 
    \begin{corollary} \label{cor:tens_simp_by_proj}
        \begin{align*}
            L_6 = L_1\otimes L_7,\\
            L_5 = L_2\otimes L_7,\\
            L_3 = L_3\otimes L_7,\\
            L_4 = L_4\otimes L_7,\\
        \end{align*}
        and
        \begin{align*}
            L_0\otimes L_4 = L_7\otimes L_4=L_4,\\
            L_1\otimes L_4 = L_6\otimes L_4 = L_3,\\
            L_3\otimes L_4 = L_1\oplus L_3\oplus L_6,\\
            L_4\otimes L_4 = L_0\oplus L_4\oplus L_7,\\
            L_2\otimes L_4 = L_5\otimes L_4 = L_2\oplus L_5.\\
        \end{align*}
    \end{corollary}
    \begin{proof}
        It is obvious that if $P,P'$ are indecomposable projective objects with simple heads $L, L'$, then $M\otimes P = P'$ implies $M\otimes L=L'$. 
        
        For the last three equalities, looking at relations \ref{eq:proj_and_simple}, we get
        $$
        [L_0]+2[L_1]+2[L_2]+2[L_3]+[L_4]+2[L_5]+2[L_6] = [P_1]+[P_3]+[P_6] =
        $$
        $$
        =[P_3\otimes L_4] = [L_2\otimes L_4] + 2[L_3\otimes L_4] + [L_4\otimes L_4] + [L_5\otimes L_4];
        $$
        $$
        2[L_0] + [L_1] + [L_3]+2[L_4]+[L_6]+2[L_7] = [P_0]+[P_4]+[P_7] = 
        $$
        $$
        = [P_4\otimes L_4] = [L_3\otimes L_4]+2[L_4\otimes L_4];
        $$
        $$
        [L_1]+2[L_2]+2[L_3]+2[L_5]+[L_6]=[P_2]+[P_5] =
        $$
        $$
        =[P_2\otimes L_4] = [L_1\otimes L_4]+2[L_2\otimes L_4]+[L_3\otimes L_4].
        $$

        Now using that $L_1\otimes L_4 = L_3, L_7\otimes L_4 = L_4,$  $L_2\otimes L_7=L_5$, so $L_5\otimes L_4= L_2\otimes L_4$, we get three linear equations:
        $$
        [L_0]+2[L_1]+2[L_2]+2[L_3]+[L_4]+2[L_5]+2[L_6]  = 2[L_2\otimes L_4] + 2[L_3\otimes L_4] + [L_4\otimes L_4],
        $$
        $$
        2[L_0] + [L_1] + [L_3]+2[L_4]+[L_6]+2[L_7] = [L_3\otimes L_4]+2[L_4\otimes L_4],
        $$
        $$
        [L_1]+2[L_2]+[L_3]+2[L_5]+[L_6] = 2[L_2\otimes L_4]+[L_3\otimes L_4].
        $$

        Solving, we get
        $$
        [L_4\otimes L_4] = [L_0]+[L_4]+[L_7],
        $$
        $$
        [L_3\otimes L_4] = [L_1] + [L_3]+[L_6],
        $$
        $$
        [L_2\otimes L_4] = [L_2]+[L_5].
        $$
        To finish the proof we note that there are no extensions between the simple modules in each of the sums above.
    \end{proof}

    The rest of the proof of part (3) is analogous to that of the Corollary \ref{cor:tens_simp_by_proj} above. The most important calculation is that of $L_1\otimes L_1$. A straightforward computation shows that 
    $$
    L_1\otimes P_0 = P_1\oplus P_1\oplus P',
    $$
    where $P'$ is some projective object with trivial projection onto the principal block. 

    As $L_1$ is self-dual (see Corollary \ref{cor:self-dual}) of categorical dimension $\dim P_0 -2\dim L_0=0-2\neq 0$, $\be=L_0$ must be a direct summand of $L_1\otimes L_1$.

    Now, let $\pi_i:P_i\to L_i$, and $\iota_i: L_i\to \mathrm{Ker}(\pi_i)$ be the projection onto the head and the embedding of the socle homomorphisms respectively for $i=0,1$. Let $\iota'_i$ be the composition of $\iota_i$ with the embedding of $\mathrm{Ker}(\pi_i)$ into $P_i$. Recall that $\mathrm{Coker}( \iota_0) = L_1$.
    
    Then
    $$
    \mathrm{Ker}(\pi_0\otimes id_{L_1})=P_1\oplus \mathrm{Ker}(\pi_1)\oplus P',
    $$
    since there is only one (up to scalar) nonzero homomorphism $P_1\to L_1$. Consider the map $\iota_0\otimes id_{L_1}$ as a homomorphism from $L_1$ to $P_1\oplus \mathrm{Ker}(\pi_1)\oplus P'$. We have two possibilities for
    $$
    L_1\otimes L_1 = \mathrm{Coker}(\iota_0\otimes id_{L_1}).
    $$
    If $L_1$ maps via $\iota_1$ to $\mathrm{Ker}(\pi_1)$ and via zero to $P_1$ then 
    $$
    L_1\otimes L_1 = P_1\oplus P'\oplus \mathrm{Coker}(\iota_1) = P_1\oplus P'\oplus L_0\oplus L_2.
    $$
    Otherwise, if the map from $L_1$ to $P_1$ is nonzero, we get
    $$
    L_1\otimes L_1 = \mathrm{Coker}(\iota'_1) \oplus \mathrm{Ker}(\pi_1) \oplus P'.
    $$
    Note that this case is impossible as $L_0$ is not a direct summand of either $\mathrm{Coker}(\iota'_1)$ or $\mathrm{Ker}(\pi_1)$.
    
    The remaining tensor products $L_i\otimes L_j$ are computed inductively, using the exact computations for $L_i\otimes P_j$ and the results of Corollary \ref{cor:tens_simp_by_proj}.  Note that as $L_3=L_1\otimes L_4$, $L_5=L_2\otimes L_7, L_6=L_1\otimes L_7$, we only need to compute  tensor products $L_1\otimes L_2$ and $L_2\otimes L_2$.

    Let $(M)_0$ denote the projection of an object $M$ onto the principal block. We get
    $$
    [(L_1\otimes L_2)_0]=[(L_1\otimes P_1)_0]-2[(L_1\otimes L_1)_0]-[L_1\otimes L_0] = 
    $$
    $$
    [P_0]+2[P_1]+[P_2]-2[P_1]-2[L_0]-2[L_2]-[L_1] = [L_1]+[L_3],
    $$
    so  $(L_1\otimes L_2)_0 =L_1\oplus L_3$, as there are no extensions between $L_1$ and $L_3$. 

    Similarly,
    $$
    [(L_2\otimes L_2)_0] = [(L_2\otimes P_1)_0]-2[(L_2\otimes L_1)_0] - [L_2\otimes L_0] =
    $$
    $$
    = [P_1]+[P_3] - 2[L_1]-2[L_3] - [L_2] = [L_0]+[L_2]+[L_4]+[L_5].
    $$
    Since vertices $0,2,4,5$ are not connected in $X_{sr}$ (see Figure \ref{fig:Xsr_G2_labeled}), we get 
    $$
    (L_2\otimes L_2)_0 = L_0\oplus L_2\oplus L_4\oplus L_5.
    $$
    This ends the proof of part (3).

(4) Recall that the M\"uger center of a category $\C$ is a full tensor subcategory $\C'$, generated by all simple objects $X\in \C$ such that $c_{Y,X}\circ c_{X,Y} = id_{X\otimes Y}$ for any $Y\in \C$ (where $c_{X,Y}$ is the braiding). Note that $\C'$ is automatically a symmetric tensor category.

We would like to show that the M\"uger center of $\C$ is the subcategory spanned by $L_0, L_4$, and $L_7$. These simple objects do not admit any extensions between them, so the category in question is semisimple. Moreover, the computations in Corollary \ref{cor:tens_simp_by_proj} show that there is a tensor equivalence between this subcategory and the category $\Rep (S_3)$, sending $L_4$ to the simple 2-dimensional representation $V$, and $L_7$ to the sign representation $sgn$ of $S_3$ (see Figure \ref{fig:E_7 simples}).

\begin{figure}[htbp]
  \centering
  \begin{tikzpicture}[scale=1.2, every node/.style={font=\small}]

\draw[black, thick] (0,0) -- (6,0);
\draw[black, thick] (3,0) -- (3,-1);

\node[draw, circle, fill=yellow, inner sep=1.5pt] at (0,0) {};

\node at (0,0.2) {$L_0=\be$};

\node[draw, circle, fill=red, inner sep=1.5pt] at (1,0) {};

\node at (1,0.2) {$L_1$};

\node[draw, circle, fill=blue, inner sep=1.5pt] at (2,0) {};
\node at (1.8,0.2) {$L_2$};
\node[draw, circle, fill=red, inner sep=1.5pt] at (3,0) {};
\node at (2.8,0.2) {$V\otimes L_1$};
\node[draw, circle, fill=yellow, inner sep=1.5pt] at (3,-1) {};
\node at (3,-1.2) {$V$};
\node[draw, circle, fill=blue, inner sep=1.5pt] at (4,0) {};
\node at (4,0.2) {$sgn\otimes L_2$};
\node[draw, circle, fill=red, inner sep=1.5pt] at (5,0) {};
\node at (5.2,0.2) {$sgn\otimes L_1$};
\node[draw, circle, fill=yellow, inner sep=1.5pt] at (6,0) {};

\node at (6.2,0.2) {$sgn$};
\end{tikzpicture}
  \caption{The action of $\Rep(S_3)$ on the principal block of $\C$}
  \label{fig:E_7 simples}
\end{figure}
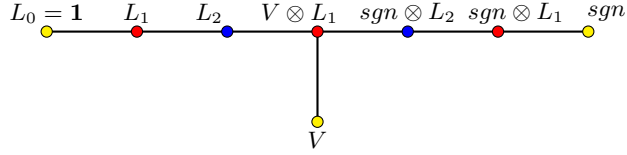

The rest of the proof relies on the following computational results.
\begin{lemma} \label{lemma:same orbit}
    For any weight $\lambda$ in $P_{A_{sr}}$,  tensor product $L_4\otimes T(\lambda)$ is the direct sum of objects $T(\mu)$ with $\lambda$ and $\mu$ in the same orbit under the dot-action of $\tilde W_l$.

    In particular, $\theta_\lambda = q^{\langle \lambda, \lambda+2\rho\rangle} = q^{\langle \mu, \mu+2\rho\rangle} = \theta_\mu$.

    If $\lambda$ is on a wall, i.e. not in the orbit of zero, then $T(\lambda)$ is a simple object of $\C$, which implies $\theta_{T(\lambda)}=\theta_\lambda=\theta_\mu = \theta_{T(\mu)}$ (see property (3) in Section \ref{qgtitling}). 
\end{lemma}
\begin{lemma} \label{lemma: integral FPdim}
    Objects $L_0,L_4, L_7$ are the only simple objects in $\C$, whose Frobenius-Perron dimensions are integer.
\end{lemma}

Since $L_4$ lies in the same block as $L_0=\be$ and $\theta\in \mathrm{Aut}(id_\C)$, we must have $\theta_{L_4}=id_{L_4}$.

Equation \ref{twist} implies that for any simple object $X$ in $\C$
$$
c_{X,L_4}\circ c_{L_4,X} = \frac{\theta_{L_4\otimes X}}{\theta_X\otimes \theta_{L_4}} = \frac{\theta_{L_4\otimes X}}{\theta_X}.
$$

Now, by Corollary \ref{cor:tens_simp_by_proj} and Lemma \ref{lemma:same orbit}, $X\otimes L_4$ is the direct sum of simple objects $Y_1,\ldots, Y_k$ satisfying $\theta_{Y_i}=\theta_X$.

Thus, $c_{X,L_4}\circ c_{L_4,X} = id_{L_4\otimes X}$, proving that $L_4$ is in $\C'$.

As $L_4\otimes L_4 = L_0\oplus L_4\oplus L_7$, we get automatically that $L_7\in \C'$.

Finally, to prove that $\C'$ does not contain any other simple objects, we note that $\C'$ is a finite symmetric tensor category, so it admits a fiber functor to super vector spaces. Thus, the Frobenius-Perron dimension of all objects in $\C'$ must be integer. We then use Lemma \ref{lemma: integral FPdim} to conclude the proof.

\end{proof}

\subsection{De-equivariantization of $\C(G_2, G_2(a_1), 7, q)$}

Let $A=Fun(S_3)$ be the algebra of functions on $S_3$ considered as an object of $\Rep(S_3)=\C'\subset \C$ via the action of $S_3$ by left translations.

Recall that the de-equivariantization of $\C$ with respect to the monoidal action of $\Rep(S_3)$  is the category 
$$
\overline \C = \overline \C(G_2, G_2(a_1), 7, q) = {_A\C},
$$
of $A$-modules in $\C$ (see for example \cite[4.1.4]{DGNO} for definition). We get the de-equivariantization (free module) functor
$$
F: \C\to \overline \C
$$
$$
X\mapsto A\otimes X,
$$
which is a braided tensor functor. We also have its right adjoint, the forgetful functor 
$$
I: \overline \C \to \C.
$$

The action of $S_3$ on $A$ by \textit{right} translations induces the action of $S_3$ on $\overline \C$ by tensor automorphisms $T_g, g\in S_3$. By general theory of equivariantization (see \cite[Proposition 4.19]{DGNO}), $\C$ is equivalent to the category $\overline \C^{S_3}$ of $S_3$-equivariant objects in $\overline \C$. Under this equivalence, functor $F$ is identified with the forgetful functor $ \overline\C^{S_3}\to \overline\C$, and functor $I$, as its right adjoint, is identified with the induction functor $X \mapsto \bigoplus_{g\in S_3} T_g(X)$. 

We get the following factorization
\begin{equation} \label{eq: FPdim of de-equiv}
 \FP (\C) = |S_3| \cdot \FP (\overline \C) = 6\cdot \FP ( \overline\C)   
\end{equation}
(see \cite[Proposition 4.26]{DGNO}).

Let us describe in more detail the module structure of $\C$ over $\Rep(S_3)$. Let us identify $L_4$ with $V$, the $2$-dimensional simple representation of $S_3$, and $L_7$ with $sgn$, the sign representation of $S_3$. The computation of matrices of multiplication by $L_4$ and $L_7$ shows the following
\begin{lemma}
    \label{lemma: Rep(S_3)-module structure}
    (1) Let $P$ be an indecomposable projective object of $\C$ with simple head $L$, and let $X$ be any object of $\Rep(S_3)\subset \C$.  Then $X\otimes L$ is semisimple, and is precisely the head of $X\otimes P$.
    
     (2) Every simple objects $L$ of $\C$ satisfies exactly one of the three following properties:
    \begin{itemize}
        \item[(a)] $L, V\otimes L, sgn\otimes L$  are simple, pairwise non-isomorphic objects of $\C$. There are $10$ such simple objects in $\C$, and they split into $5$ pairs of the form $\{L,sgn\otimes L\}$. In particular, $4$ of these objects are in the principal block. Namely, pairs $\{L_0, L_7\}$ and $\{L_1, L_6\}$.

        If $P$ is the projective cover of $L$ then $V\otimes P, sgn\otimes P$ are projective covers of $V\otimes L$ and $sgn\otimes L$ respectively.

        \item[(b)] $L=V\otimes L'$ for some simple $L'$, and so $sgn\otimes L = L,  V\otimes L = V\otimes V\otimes L'  = L'\oplus L\oplus (sgn\otimes L')$. There are $5$ such simples in $\C$, and all of them are of the form $V\otimes L'$, where $L'$ satisfies property (a). Out of the $5$ objects, $2$ are in the principal block: $L_4 = V\otimes L_0,$ and $L_3=V\otimes L_1$. 

        \item[(c)] $sgn\otimes L = L'$ is a simple object, non-isomorphic to $L$, and $V\otimes L = L\oplus L'$. There are $8$ such simple objects, coming in pairs $\{L, sgn\otimes L\}$. Out of them, there are $2$ objects in the principal block: $\{L_2, L_5\}$.
    \end{itemize}
\end{lemma}
\begin{proof}
    When $L$ is not in the principal block, and thus projective, the result follows from the the computation of the matrix of multiplication of objects $\{T(\mu_i)~|~\mu_i\in P_{A_{sr}}\}$ by $L_4$ and $L_7$ (see the proof of part (3) of Theorem \ref{main 7} in Section \ref{{section: proof of theorem 1.3}}).

    When $L$ is in the principal block, the result follows from Corollary \ref{cor:tens_simp_by_proj}.
\end{proof}

Let us now describe the simple and projective objects of $\overline \C={_A\C}$.

\begin{lemma} \label{lemma:semisimple A-modules}
     Let $\C^{ss}\subset \C$ be the subcategory of semisimple objects. It follows from Lemma \ref{lemma: Rep(S_3)-module structure}, that $\C^{ss}$ is a module category over $\Rep(S_3)\subset \C^{ss}$. 
     
     Then the category $_A\C^{ss}$ of $A$-modules in $\C^{ss}$ is semisimple. 
\end{lemma}
\begin{proof}
    The proof is analogous to the proof of \cite[Proposition 7.8.30]{EGNO}. Namely, since $A=Fun(S_3)$ is separable (so there is a $A$-$A$-linear splitting $A\to A\otimes A$ of the multiplication map),  any $ M = A\otimes_{A} M \in {_A\C^{ss}}$ is a direct summand of the free module $A\otimes M = A\otimes A\otimes_A M$. Since $M$ is projective in $\C^{ss}$, we get that $A\otimes M$ is projective in ${_A\C^{ss}}$, and hence $M$ is projective in ${_A\C^{ss}}$.
\end{proof}

\begin{lemma} \label{lemma: A-module head}
    If $P$ is a projective object of $\C$ with head $L$ then $A\otimes P$ is a projective object of $\overline \C = {_A\C}$ with head $A\otimes L$.
\end{lemma}
\begin{proof}
    By Lemma \ref{lemma:semisimple A-modules}, $A\otimes L$ is a semisimple $A$-module. It follows that the head of $A\otimes P$ in ${_A\C}$ is $A\otimes L\oplus M$ for some $A$-module $M$. However,  Lemma \ref{lemma: Rep(S_3)-module structure} implies that $A\otimes L$ is the head of $A\otimes P$ in $\C$, and thus the head of $M$ in $\C$ is zero, and therefore $M$ is zero.  
\end{proof}

\begin{corollary}\label{cor: simple A-modules}
    Every simple $A$-module is a summand of $A\otimes L$ for some simple object $L\in\C$.
\end{corollary}
\begin{proof}
    Let $Y$ be a simple $A$-module with projective cover $Z\in{_A\C}$. Let $P\in \C$ be a projective object with an epimorphism $p:P\to Z$. Then the $A$-linear extension $A\otimes P\to Z$ of $p$ is an epimorphism of projective $A$-modules, and thus $Z$ is a direct summand in $A\otimes P$. It follows that $Y$ is a direct summand of $A\otimes L$, where $L$ is some simple component of the head of $P$.
\end{proof}

\begin{lemma} \label{lemma: forgetful}
(1)    Let $L$ be a simple object of $\C$. Suppose $A\otimes L = Y_1\oplus\ldots Y_k$, where $Y_i$ are simple, pairwise non-isomorphic $A$-modules. Then
    $$
    I(Y_i)  = I(Y_j) \text{ for all }i,j,
    $$
    where $I:{_A\C}\to \C$ is the forgetful functor.

(2) For each $i$, the $S_3$-orbit of $Y_i$ (under the action by $T_g, g\in S_3$) is the set $\{Y_1,\ldots, Y_k\}$.
\end{lemma}
\begin{proof}
  (1)  We have $\Hom_\C(L', I(Y_i)) = \Hom_A(A\otimes L', Y_i)$ for any simple object $L'\in\C$ and any $i=1,\ldots, k$. 

    On the other hand, by Lemma \ref{lemma: Rep(S_3)-module structure}, $\Hom_A( A\otimes L', A\otimes L) = \Hom_\C(L', A\otimes L) \neq 0$ only if $L'=sgn\otimes L$ or $L'=V\otimes L$. In the former case, $A\otimes L' = A\otimes L$, and in the latter $A\otimes L'=(A\otimes L)^{\oplus 2}$. In both cases, the dimensions of $\Hom_A(A\otimes L', Y_i)$ are the same for all $i$. 

    Since $A\otimes L$, and hence $I(Y_i)$, is semisimple, this ends the proof.

    (2) This statement follows from the identification of $I$ with the induction functor $\overline \C \to \overline\C^{S_3}$ and part (1).
    
\end{proof}

\begin{corollary} \label{cor: forgetful}
    Let $A\otimes L = Y_1\oplus\ldots \oplus Y_k$ be as in Lemma \ref{lemma: forgetful}. Let $Z_i$ be the projective cover of $Y_i$ in ${_A\C}$. Then $I(Z_i)=I(Z_j)$ for all $i,j$.
\end{corollary}
\begin{proof}
    Note that $A\otimes P$ is the projective cover of $A\otimes L$ both in $\C$ and in ${_A\C}$ (see Lemmas \ref{lemma: Rep(S_3)-module structure}(1) and \ref{lemma: A-module head}). Thus $I(Z_i)$ is the projective cover of $I(Y_i)$ in $\C$. The statement then follows from Lemma \ref{lemma: forgetful}.
\end{proof}

\subsection{Proof of Theorem \ref{main 7mod}}
Let us recall the statement of Theorem \ref{main 7mod}:
\begin{itemize}
\item[(1)] The category $\bar \C(G_2,G_2(a_1),7,q)$ has 12 trivial blocks and one block of type $\tilde D_4$.
In particular it has 17 simple objects.

\item[(2)] We have $\FP(\bar \C(G_2,G_2(a_1),7,q))=49(7+15[3]_7+12[5]_7)\approx 3054.068811$.

\item[(3)] The category $\bar \C(G_2,G_2(a_1),7,q)$ has stable Chevalley property.

\item[(4)] The category $\bar \C(G_2,G_2(a_1),7,q)$ is completely anisotropic: it has no non-trivial
commutative exact algebras.
\end{itemize}
\begin{proof}
    (1) Recall that $A=Fun(S_3)$, so $A =\be \oplus V^{\oplus 2}\oplus sgn = L_0\oplus L_4^{\oplus 2} \oplus L_7$ as an object of $\Rep(S_3)\subset \C$. We have the following isomorphisms of $A$-modules in $\Rep(S_3)$, and hence in $\C$:
    $$
    A\otimes sgn  = A,~~A\otimes V = A\oplus A.
    $$

    Let $L$ be a simple object of $\C$. Then $$
    \End_A (A\otimes L) = \Hom_{\C}(L,A\otimes L)=
    $$
    $$
    =\Hom_{\C}(L, L\oplus (V\otimes L)^{\oplus 2} \oplus (sgn\otimes L)).
    $$

    If $L$ satisfies property (a) of Lemma \ref{lemma: Rep(S_3)-module structure}, part(1), then 
    $$
    \dim \End_A(A\otimes L) = 1,
    $$
    so, by Lemma \ref{lemma:semisimple A-modules}, $Y(L):=A\otimes L \in {_A\C^{ss}}$ is irreducible. Note also that $A\otimes L$ and $A\otimes L'$ are isomorphic as $A$-modules if and only if $L=\epsilon\otimes L'$ for some $1$-dimensional representation of $S_3$, i.e. for $L=L'$, or for $L=sgn\otimes L'$. Therefore, we get $5$ isomorphism classes of simple objects in $\overline \C$ coming from such simple objects $L\in\C$.

    If $L$ satisfies property (b), so that $L=V\otimes L'$, then $A\otimes L = A\otimes V\otimes L' = (A\oplus A)\otimes L'  = (A\otimes L')^{\oplus 2}$.

    Finally, if $L$ satisfies property (c), we get 
    $$
    \dim\End_A(A\otimes L)=3,
    $$
    so, by Lemma \ref{lemma:semisimple A-modules},  $Y(L)=A\otimes L$ is the direct sum of three pairwise non-isomorphic simple modules $Y_1(L),Y_2(L),Y_3(L)$.

    Recall that there are $4$ pairs $\{L,sgn\otimes L\}$ of simple modules satisfying (c), and that $A\otimes L = A\otimes sgn\otimes L$.    
    Since $\Hom_A(A\otimes L, A\otimes L')=\Hom_\C(L, A\otimes L')$ is nonzero only if $L'=L$ or $L'=sgn\otimes L$, we get $3\cdot 4=12$ isomorphism classes of simple $A$-modules this way. 

    By Corollary \ref{cor: simple A-modules}, all simple objects of $\overline \C$ are constructed this way. Thus, we get $5+12=17$ isomorphism classes of simple objects in $\C$.

    Now, let us describe the blocks in $\C$.

    Note first, that if $P$ is a simple projective object of $\C$ then $Y(P)=A\otimes P$ is projective and semisimple in $\overline \C$, and thus its simple summands lie in (distinct) trivial blocks of $\overline \C$. By Lemma \ref{lemma: Rep(S_3)-module structure}, we get $(5-2)+3\cdot (4-1) = 12$ such  projective simple objects in $\overline \C$.

    Now, let us describe what happens for objects in the principal block. We get simple $A$-modules $Y(L_0)=A=\be_{\overline\C}, Y(L_1),$ and $Y_i(L_2)$ for $i=1,2,3$. Let us denote their projective covers by $Z(L_0), Z(L_1),$ and $Z_i(L_2)$ respectively. We have, by Lemma \ref{lemma: A-module head},  $Z(L_0)=A\otimes P_0, Z(L_1)=A\otimes P_1$, and $Z_1(L_2)\oplus Z_2(L_2)\oplus Z_3(L_2) = A\otimes P_2$.

    For every $A$-module $M\in \overline \C$, let $[M]$ denote its class in the Grothendieck ring of $\overline \C = {_A\C}$.
    Combining the above with decomposition \ref{eq:proj_and_simple}, we get
    $$
    [Z(L_0)] = 2[Y(L_0)]+[Y(L_1)],
    $$
    $$
    [Z(L_1)] = [Y(L_0)]+2[Y(L_1)]+\sum_{i=1}^3 [Y_i(L_2)],
    $$
    $$
    [A\otimes P_2] = [Y(L_1)]+2[Y(L_2)]+[A\otimes L_3] = [Y(L_1)]+\sum_{i=1}^3 2[Y_i(L_2)]+2[Y(L_1)],
    $$
    since $A\otimes L_3 = (A\otimes L_1)^{\oplus 2}$. Using Corollary \ref{cor: forgetful} and Lemma \ref{projubf}, we get
    $$
    [Z_i(L_2)]  =[Y(L_1)] + 2[Y(L_2)].
    $$

    We conclude that the principal block of $\overline \C$ has type $\tilde D_4$ (see Figure \ref{fig:D_4}, compare with Figure \ref{fig:E_7 simples}).
    
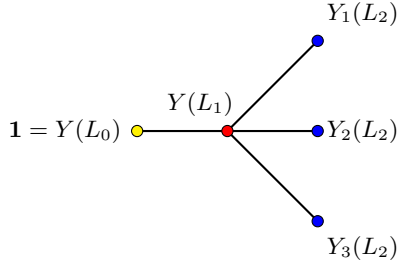
\begin{figure}[htbp]
  \centering
  \begin{tikzpicture}[scale=1.2, every node/.style={font=\small}]

\draw[black, thick] (8,0) -- (10,0);
\draw[black, thick] (9,0) -- (10,1);
\draw[black, thick] (9,0) -- (10,-1);

\node[draw, circle, fill=yellow, inner sep=1.5pt] at (8,0) {};

\node at (7.2,0) {$\be = Y(L_0)$};

\node[draw, circle, fill=blue, inner sep=1.5pt] at (10,0) {};
\node at (10.5,0) {$Y_2(L_2)$};

\node[draw, circle, fill=blue, inner sep=1.5pt] at (10,1) {};
\node at (10.5,1.3) {$Y_1(L_2)$};

\node[draw, circle, fill=blue, inner sep=1.5pt] at (10,-1) {};
\node at (10.5,-1.3) {$Y_3(L_2)$};

\node[draw, circle, fill=red, inner sep=1.5pt] at (9,0) {};
\node at (8.7,0.3) {$Y(L_1)$};
 
\end{tikzpicture}
  \caption{Graph $X$ for the principal block of $\overline \C$}
  \label{fig:D_4}
\end{figure}

(2) The formula follows from part (2) of Theorem \ref{main 7} and Equation \ref{eq: FPdim of de-equiv}.

(3) Recall that by part (3) of Theorem \ref{main 7}, the tensor product $L\otimes L'$ of two simple objects in $\C$ is the direct sum of a semisimple object with a projective object.

By Lemma \ref{lemma:semisimple A-modules}, $A\otimes L$ is semisimple in $\overline C$ if $L$ is simple in $\C$, and for $P\in \C$ projective, $A\otimes P$ is projective in $\overline \C$. We get that  $(A\otimes L)\otimes_A (A\otimes L') = A\otimes(L\otimes L')$ is the direct sum of a semisimple and a projective $A$-module.  

Since, by Corollary \ref{cor: simple A-modules}, all simple objects of $\overline \C$ are direct summands of $A\otimes L$ for $L\in \C$ simple, we get the desired statement. 
    
(4) The proof of part (4) relies heavily on computations of Frobenius-Perron dimensions of objects in $\overline \C$. Note that the de-equivariantization functor $F: \C \to \overline \C$ is a tensor functor, and thus preserves the Frobenius-Perron dimensions and commutes with the twist morphisms. By part (2) of Lemma \ref{lemma: forgetful}, if $L\in\C$ is simple, and $F(L)=Y_1\oplus \ldots Y_k$ is the decomposition into simple summands, then $\FP(Y_i) = \frac{1}{k} \FP (L)$. 

The above, together with the classification of all simple objects in $\overline \C$ and the computation of Frobenius-Perron dimensions of objects in $\C$ leads to the following result.

\begin{lemma} \label{lemma: de-eq dims and twists}
    (1) The only simple object of $\overline \C$ with integral Frobenius-Perron dimension is $\be$. Thus, $\overline \C$ has no nontrivial  invertible objects and the M\"uger center $\overline\C'$ of $\overline \C$ is trivial, i.e. $\overline \C$ is non-degenerate. 

    (2) The only simple objects of $\overline \C$ on which the twist $\theta$ acts by identity are $Y(L_0), Y(L_1), Y_1(L_2), Y_2(L_2), Y_3(L_2)$, that is, the simple objects in the principal block of $\overline\C$.

    (3) Let $\alpha = [3]_7+[5]_7.
    $ The Frobenius-Perron dimensions of these simple objects are the following
    $$
    \FP(Y(L_0))=\FP(\be) = 1,
    $$
    $$
    \FP(Y(L_1)) = 2+3\alpha,
    $$
    $$
    \FP(Y_i(L_2))= \alpha,~~ \text{ for } i =1,2,3.
    $$
\end{lemma}

\begin{proof}
    For (1) and (3) we use our computation of $\FP(L)$ for simple objects $L\in \C$. For each simple object $Y\in \overline \C$, we get $\FP(Y)=\FP(L)$ if $Y=A\otimes L$, and $\FP(Y)=\frac{1}{3}\FP(L)$ if $Y$ is a proper summand in $A\otimes L$.

    For (2), note first that $\theta$ acts by identity on all simple objects in the principal block. We then use formula \ref{twist} to compute the action of $\theta$ on $T(\lambda)\in \C$ outside of principal block. Since we know that the action of $\theta$ on $T(\lambda)$ depends only on the orbit of $\lambda$ under the dot-action of $\tilde W_l$, we only need to compute $(\lambda, \lambda +2\rho)$ for $6$ weights $\lambda$. Namely, $\lambda = \omega_1, \omega_2, 3\omega_1, 2\omega_2, 2\omega_1+\omega_2, $ and $4\omega_1$ (see Figure \ref{fig:weights_in_cell}). We use the normalization $(\omega_1,\omega_1)=2, (\omega_1, \omega_2)=3, (\omega_2,\omega_2) = 6$.  
    
\end{proof}

By Corollary \ref{red} and part (1) of Lemma \ref{lemma: de-eq dims and twists} above, we just need to show that there are
no commutative exact algebras $R\in \overline \C$ such that all simple subquotients of $R$ are in the principal block. By part (3) of Lemma \ref{lemma: de-eq dims and twists}, 
this implies that $\FP(R)=r+s\alpha$ for some $r,s\in \Z_{\ge 0}$ (where $\alpha = [3]_7+[5]_7$).  

On the other hand, it is known (see \cite[Lemma 5.17]{ShYa}) that for an indecomposable commutative exact algebra $R$ in a non-degenerate category $\mathcal D$ we have 
$$\FP(\mathcal D_R^{loc})=\frac{\FP(\mathcal D)}{\FP(R)^2}.$$
Here $\FP(\mathcal D_R^{loc})$ is the Frobenius-Perron dimension of some category, so
this number is an algebraic integer, $\ge 1$, and it is greater than any of its
algebraic conjugates. 

By part (2) of Theorem \ref{main 7mod} and a direct computation, the norm of $\FP(\overline \C)$ is $7^7$.
Thus we have the following properties of $\FP(R)$:
\begin{itemize}
    \item  $\FP(R)=r+s\alpha$ for some $r,s\in \Z_{\ge 0}$;

    \item  $\FP(R)\le \sqrt{\FP(\overline \C)}\approx 55.2636$;

    \item Norm of $\FP(R)$ is some integral power of $7$ up to sign.
\end{itemize}

It is easy to find all the numbers satisfying the properties above using a computer.
Here is the list:

$$1,7,49,\alpha, 1+\alpha, 1+2\alpha, 1+3\alpha, 2+3\alpha, 3+4\alpha, 2+5\alpha, 7\alpha, 7+7\alpha.$$

The number $1$ can serve only as dimension of the trivial algebra $R=\be$.
The algebras $R$ with $\FP(R)=7, 49$ are not indecomposable (since $R$ would be a direct sum of several copies of $\be$ in this case, see \cite[Theorem 4.4.1]{EGNO}).
An algebra $R$ with $\FP(R)=\alpha$ or $7\alpha$ can't have a unit. For the remaining
possibilities for $\FP(R)$ the number $\frac{\FP(\bar \C)}{\FP(R)^2}$
is not the largest among its algebraic conjugates (for example when
$\FP(R)=3+4\alpha$ we get 
$\frac{\FP(\bar \C)}{\FP(R)^2}\approx 8.2884$,
however this number has a conjugate $\approx 328.5234$). This completes the proof.

\end{proof}

\section{Commutative exact algebras and Witt group}\label{Witt}
In this section $\K$ is an algebraically closed field of arbitrary characteristic.
\subsection{Exact algebras} Let $\C$ be a finite tensor category over $\K$ and let $A=(A,m,u)\in \C$ be an algebra,
see e.g. \cite[Definition 7.8.1]{EGNO}. One defines notions of right and left modules
over $A$ in $\C$, see \cite[Definition 7.8.5]{EGNO}. Let $\C_A$ (respectively $\phantom{}_A\C$)
be the category of right (respectively left) $A-$modules. Then $\C_A$ has a structure of module
category over $\C$, see \cite[Proposition 7.8.10]{EGNO}. An algebra $A\in \C$ is {\em exact}
if the module category $\C_A$ is exact (see \cite[Definition 7.8.20]{EGNO}); this means
that for any $M\in \C_A$ and any projective object $P\in \C$, the $A-$module $P\otimes M\in \C_A$
is a projective object of category $\C_A$. Equivalently, tensor product $\otimes_A$ over $A$
is exact, or the category $\phantom{}_A\C_A$ of $A-$bimodules is rigid, see e.g. 
\cite[Theorem 5.1]{ShYa}. A very useful characterization of exact algebras was given
in \cite{CSZ}:

\begin{theorem}{\cite[Theorem 7.1]{CSZ}}\label{CoSZ}
For an algebra $A$ in a finite tensor category $\C$ the following conditions are equivalent:

(1) $A$ is exact;

(2) $A$ has no non-zero nilpotent ideals;

(3) $A$ is a finite direct product of simple algebras.
\end{theorem}

\subsection{Graded algebras} Let $G$ be a group. We say that an algebra $A\in \C$ is
$G-$graded if it is equipped with decomposition $A=\oplus_{g\in G}A_g$ and the image $A_gA_h$ 
of the multiplication map $m: A_g\otimes A_h\to A$ is contained in $A_{gh}\subset A$ for all
$g,h \in G$. Let $e\in G$ be the identity element. Clearly $A_e\subset A$ is a subalgebra of $A$ (note that the image of the identity morphism
is automatically contained in $A_e$).

Let $\phantom{}_A\C(G)$ be the category of left $G-$graded $A-$modules. 
We say that $A$ is {\em strongly graded} by $G$ if $A$ is $G-$graded and, in addition, $A_gA_h=A_{gh}$
for all $g,h\in G$. 

\begin{lemma}\label{dade1}
 (1) $A$ is strongly graded by $G$ if and only if $A_gA_{g^{-1}}=A_e$
 for all $g\in G$.

 (2) Assume $A$ is strongly graded by $G$ and let $K=\oplus_{g\in G}K_g\in \phantom{}_A\C(G)$. Then $K_e=0$ implies $K=0$.
\end{lemma}

\begin{proof}
  (1) The condition  $A_gA_{g^{-1}}=A_e$ is a part of definition of strongly graded algebra, so one implication is clear. Now assume
  that $A_gA_{g^{-1}}=A_e$. Then $A_{gh}=A_eA_{gh}=A_gA_{g^{-1}}A_{gh}\subset A_gA_h$ and we proved the second implication.

  (2) Assume $K_e=0$. Then $K_g=A_eK_g=A_gA_{g^{-1}}K_g\subset A_gK_e=0$.
\end{proof}

For a $G-$graded algebra $A$ we have functors $(?)_e: \phantom{}_A\C(G)\to \phantom{}_{A_e}\C, K\mapsto K_e$ and $A\otimes_{A_e}?: \phantom{}_{A_e}\C \to \phantom{}_A\C(G), M\mapsto \oplus_{g\in G}A_g\otimes_{A_e}M$. 
The following result is a categorical version of a well known theorem of Dade \cite{D}.

\begin{theorem}\label{dade2}
Assume an algebra $A\in \C$ is strongly graded by a finite group $G$. 

(1) The functors $(?)_e$ and $A\otimes_{A_e}?$ are mutually inverse equivalences.

(2) The multiplication induces isomorphisms of $A_e-$bimodules
$A_g\otimes_{A_e}A_h\simeq A_{gh}$.
\end{theorem}

\begin{proof}
    (1) It is clear that $(A\otimes_{A_e}M)_e=A_e\otimes_{A_e}M=M$, so
    one composition is isomorphic to the identity functor. Now let
    $K=\oplus_{g\in G}K_g\in \phantom{}_A\C(G)$. The action morphism
    $A\otimes K\to K$ restricted to $K_e$ induces a morphism in $\phantom{}_A\C(G)$ $A\otimes_{A_e}K_e\to K$. It is clear that this
    morphism is isomorphism in the $e-$component. Thus its kernel and
    cokernel have zero $e-$component. It follows from Lemma \ref{dade1} (2) that the kernel and cokernel are both zero and the morphism 
    $A\otimes_{A_e}K_e\to K$ is an isomorphism. Thus the second composition is also isomorphic to the identity functor.

    (2) We proved in (1) that the action map induces isomorphism
    $A\otimes_{A_e}K_e\to K$ for any $K\in \phantom{}_A\C(G)$.
    Now apply it for $K=A(h)$ where $A(h)_g=A_{gh}$.
\end{proof}

\begin{remark}
    The proofs of Lemma \ref{dade1} and Theorem \ref{dade2} follow
    \cite[1.3]{NO} very closely.
\end{remark}

The following result is an immediate consequence of Theorem \ref{dade2} (2):

\begin{corollary}\label{dade3}
Assume that an algebra $A\in \C$ is strongly graded by group $G$ and 
$A_e=\be$. Then each $A_g$ is an invertible object of $\C$.
\end{corollary}

\subsection{Commutative algebras} In this section we assume that $\C$ is a braided finite
tensor category. Thus it makes sense to talk about commutative algebras $A\in \C$.

\begin{proposition}\label{subex}
   Assume $A\in \C$ is a commutative exact algebra. Then

   (1) Any unital subalgebra $B\subset A$ is exact.

   (2) If $A$ is indecomposable then it has no non-trivial right or left ideals.
\end{proposition}

\begin{proof}
    (1) If $B\subset A$ is a subalgebra that is not exact, then $B$ contains a nontrivial
    nilpotent ideal $I$ by Theorem \ref{CoSZ}. Then $AI\subset A$ is a nontrivial nilpotent
    ideal in $A$ and we have a contradiction with Theorem \ref{CoSZ}.

    (2) Immediate from Theorem \ref{CoSZ} since right, left, and two-sided ideals coincide in $A$.
\end{proof}

\begin{remark}
    (1) One can drop assumption that $B$ is unital in Proposition \ref{subex} (1).
    Namely for the proof above it is sufficient that $B$ has its own identity (which
    might be different from the identity of $A$). Now for any subalgebra $B$ we can
    consider a bigger subalgebra $\tilde B$ spanned by $B$ and the image of the unit
    morphism of $A$. The subalgebra $\tilde B$ is unital, so it is exact; also $B\subset \tilde B$
    is an ideal. Thus by Theorem \ref{CoSZ} $B$ is a direct summand of exact algebra $\tilde B$;
    in particular it has its own unit morphism. 

    (2) Let us assume that $\mbox{char}~ \K=0$ and $\C$ is a braided fusion category. Then commutative exact algebras in $\C$ are \'etale algebras, as defined in \cite[Definition 3.1]{DMNO}.
    In this case Proposition \ref{subex} (1) positively answers question raised
    in \cite[Remark 3.4]{DNO}: any subalgebra of \'etale algebra is \'etale.
\end{remark}

Recall that a commutative exact algebra $A$ is indecomposable if and only if
$\Hom(\be,A)$ is one dimensional, see e.g. \cite[Lemma 5.8]{ShYa}.

\begin{proposition}\label{exstr}
    Let $A\in \C$ be an indecomposable commutative exact algebra. Assume that $A$ is
    $G-$graded for some group $G$. Let $H\subset G$ be a subgroup generated by all
    $g\in G$ such that $A_g\ne 0$. Then $A$ is strongly graded by $H$ (in particular
    $A_g\ne 0$ for any $g\in H$). Also $A_e\subset A$ is an indecomposable commutative exact
    algebra.
\end{proposition}

\begin{proof}
    Let $H_0=\{ g\in G\; |\; A_g\ne 0\}$. Let us show that $H_0=H$, i.e. $H_0$ is a subgroup
    of $G$. Let us show that $H_0$ is closed under multiplication. 
    Otherwise we can find $g,h\in H_0$ such that $gh\not \in H_0$. It follows
    that $AA_g\subset A$ is a proper ideal since it satisfies $AA_g\cdot A_h\subset AA_{gh}=0$.
    This is a contradiction with Proposition \ref{subex} (2). Thus, $H_0\subset G$ is a finite
    subset closed under multiplication; it follows that $H_0=H$ is a subgroup.

    Now by Proposition \ref{subex} (1), $A_e\subset A$ is an exact subalgebra; this subalgebra
    is indecomposable since $\Hom(\be,A_e)\subset \Hom(\be,A)=\K$. The image of multiplication
    $A_gA_{g^{-1}}\subset A_e$ is an ideal of $A_e$; therefore, either $A_gA_{g^{-1}}=A_e$ or
    $A_gA_{g^{-1}}=0$. In the latter case $AA_g$ is a proper ideal of $A$ since it satisfies
    $AA_g\cdot A_{g^{-1}}=0$. This is impossible by Proposition \ref{subex} (2).
    Thus $A_gA_{g^{-1}}=A_e$ for all $g\in H$ and $A$ is strongly graded by $H$ by 
    Lemma \ref{dade1} (1).
\end{proof}

\subsection{Commutative algebras in ribbon categories} In this section we assume 
that $\C$ is a ribbon finite tensor category with a twist $\theta$. For any $M\in \C$
and $t \in \K^\times$ let $M_t^\theta \subset M$ be the largest subobject such
that $\theta_M-t\Id$ acts nilpotently on $M_t^\theta$. Then we have a decomposition
$M=\oplus_{t\in \K^\times}M_t^\theta$.

\begin{remark}
    It was proved by Etingof \cite{Ev} that in characteristic zero $M_t^\theta \ne 0$ only when $t$ is a root of unity.
\end{remark}

\begin{lemma}\label{tdec}
Let $A\in \C$ be a commutative algebra. Then decomposition $A=\oplus_{t\in \K^\times}A_t^\theta$
is a grading of $A$ by group $\K^\times$.
\end{lemma}

\begin{proof}
    Let $m:A \otimes A\to A$ be the multiplication map. Using \eqref{twist}
    we get
    $$\theta_{A\otimes A}=c_{A,A}\circ c_{A,A}\circ (\theta_A\otimes \theta_A).$$
    Composing with $m$ and using commutativity and naturality of $\theta$ we get
    $$\theta_A\circ m=m\circ \theta_{A\otimes A}=m\circ (\theta_A\otimes \theta_A),$$
    whence 
    $$\theta_A^n\circ m=m\circ (\theta_A^n\otimes \theta_A^n)\;
    \mbox{for any}\; n\in \Z_{\ge 0}.$$
    It is clear that $\theta_A\otimes \theta_A-ts\Id$ acts nilpotently on $A_t^\theta \otimes A_s^\theta \subset A\otimes A$; it follows that $\theta_A-ts\Id$ acts nilpotently on
    $m(A_t^\theta \otimes A_s^\theta)$.
\end{proof}

\begin{remark}
    We don't know any examples of commutative exact algebras $A$ in ribbon finite tensor
    categories with $\theta_A^2\ne \Id$. In particular, we don't know examples of such algebras
    with $A_t^\theta \ne 0$ for $t\ne \pm 1$.
\end{remark}

\begin{corollary}\label{red}
    Assume that $\C$ is a ribbon finite tensor category such that $\C$ has no 
    nontrivial commutative exact algebras $A=A_1^\theta$ and that $\C$ has
    no nontrivial invertible objects. Then $\C$ is completely anisotropic, i.e.
    it has no nontrivial commutative exact algebras.
\end{corollary}

\begin{proof}
    Let $B\in \C$ be an indecomposable commutative exact algebra. Then 
    $B=\oplus_{t\in \K^\times}B_t^\theta$ is a $\K^\times -$grading of $B$ by Lemma \ref{tdec}.
    By Proposition \ref{exstr} this grading induces strong grading by a suitable
    subgroup $H\subset \K^\times$ and $B_1^\theta \subset B$ is an indecomposable exact subalgebra. Thus by the assumptions we have $B_1^\theta =\be$ and by Corollary
    \ref{dade3} $B_t^\theta$ is a nontrivial invertible object of $\C$ for $1\ne t\in \K^\times$. Hence $H$ is trivial and the result follows.
\end{proof}

\subsection{Commutative exact algebras in Drinfeld centers}
Let $\A$ be a finite tensor category and let $\cZ(\A)$ be its Drinfeld center (see e.g. \cite[7.13]{EGNO}). Then $\cZ(\A)$ is a finite tensor category which has a natural braiding
(see \cite[8.5]{EGNO}); moreover $\cZ(\A)$ is factorizable, see \cite[8.6]{EGNO}. 

Let $F:\cZ(\A)\to \A$ be the forgetful functor, and let $I: \A \to \cZ(\A)$ be its right 
adjoint. The object $I(\be)\in \cZ(\A)$ has a natural structure of indecomposable
commutative exact algebra, see e.g. \cite[Remark 6.10]{ShYa}; moreover this is a {\em lagrangian} algebra (an indecomposable commutative exact algebra $A$ in a factorizable
braided finite tensor category $\C$ is lagrangian if $\FP(A)^2=\FP(\C)$). Converse
statement is a combination of \cite[Corollary 5.14 (a)]{ShYa} and \cite[Corollary 5.18]{ShYa}:

(a) if a factorizable braided finite tensor category $\C$ contains a lagrangian algebra $A$
then $\C \simeq \cZ(\A)$ for some finite tensor category $\A$; in fact we can take $\A$ to
be the opposite category of $\C_A$.

Now let $\Delta \in \A$ be an indecomposable exact algebra. Then the category
$\phantom{}_\Delta \A_\Delta$ of $\Delta-$bimodules in $\A$ is a finite tensor category,
see e.g. \cite[Remark 7.12.5]{EGNO}. Moreover we have a canonical braided equivalence
$\cZ(\A)\simeq \cZ(\phantom{}_\Delta \A_\Delta)$, see \cite[Corollary 7.16.2]{EGNO}. 
Using this equivalence we can define
the forgetful functor $F_\Delta :\cZ(\A)\simeq \cZ(\phantom{}_\Delta \A_\Delta)\to \phantom{}_\Delta \A_\Delta$ and its right adjoint $I_\Delta : \phantom{}_\Delta \A_\Delta \to
\cZ(\phantom{}_\Delta \A_\Delta)\simeq \cZ(\A)$ and lagrangian algebra $I_\Delta(\be)\in \cZ(\A)$.

\begin{theorem}\label{lagr}
    Let $\A$ be a finite tensor category and let $A\in \cZ(\A)$ be an indecomposable
    commutative exact algebra. Then there exists an indecomposable exact algebra 
    $\Delta \in \A$ such that $A$ is isomorphic to a subalgebra of Lagrangian algebra
    $I_\Delta(\be)$.
\end{theorem}

\begin{proof}
    The object $F(A)\in \A$ is an algebra, possibly non-exact. Choose a maximal two-sided
    ideal $I\subset F(A)$ and set $\Delta =F(A)/I$. The algebra $\Delta$ is an indecomposable
    exact algebra in $\A$ by Theorem \ref{CoSZ}. 

    The functor $F_\Delta :\cZ(\A)\to \phantom{}_\Delta \A_\Delta$ can be described as
    follows: for any $Z\in \cZ(\A)$ we have $F_\Delta(Z)=Z\otimes \Delta$ with an obvious structure of $\Delta-$bimodule, see \cite[Remark 7.16.3]{EGNO}. In particular $F_\Delta(A)=A\otimes \Delta$. Using multiplication in $A$ we get morphism $A\otimes \Delta \to \Delta$, equivalently $d: F_\Delta(A)\to \be$. 
    It is clear that this is a homomorphism of algebras in 
    the category $\phantom{}_\Delta \A_\Delta$; in particular $d\ne 0$.

    Let $\tilde d: A\to I_\Delta(\be)$ be the image of $d$ under
    the natural isomorphism $\Hom(F_\Delta(A),\be)\simeq \Hom(A,I_\Delta(\be))$. By definition the morphism $\tilde d$
    is the composition $A\to I_\Delta(F_\Delta(A))\to I_\Delta(\be)$
    where the first map is the adjunction and the second map is $I_\Delta(d)$. Recall that $I_\Delta$ is a lax monoidal functor,
    so it sends algebras to algebras (in particular, this is how the
    structure of algebra on $I_\Delta(\be)$ was defined), and homomorphisms of algebras
    to homomorphisms of algebras. Thus $I_\Delta(d)$ is a homomorphism 
    of algebras. Also the adjunction map $A\to I_\Delta(F_\Delta(A))$
    is a homomorphism of algebras by the properties of monoidal adjunctions, see \cite[Proposition 2.1]{cafe}. Hence $\tilde d: A\to I_\Delta(\be)$
    is a homomorphism of algebras; since $A$ is simple (see Proposition \ref{subex} (2)) this homomorphism is injective.
\end{proof}

\begin{corollary}\label{8.12.3}
    Let $\A$ and $\B$ be two finite tensor categories such that $\cZ(\A)$ and 
    $\cZ(\B)$ are equivalent as braided tensor categories. Then there exists
    an indecomposable commutative exact algebra $\Delta \in \A$ and an equivalence
    of tensor categories $\B \simeq \phantom{}_\Delta \A_\Delta$. In other words,
    $\A$ and $\B$ are Morita equivalent.
\end{corollary}

\begin{proof}
    We apply Theorem \ref{lagr} to algebra $A=I_\B(\be)\in \cZ(\B)\simeq \cZ(\A)$ 
    where $I_\B$ is the right adjoint of the forgetful functor $F_\B: \cZ(\B)\to \B$.
    Thus $I_\B(\be)$ is isomorphic to some subalgebra of $I_\Delta(\be)$. This subalgebra
    must be all of $I_\Delta(\be)$ since both $I_\B(\be)$ and $I_\Delta(\be)$ are
    lagrangian subalgebras, so $\FP(I_\B(\be))=\FP(I_\Delta(\be))$. 
    Thus we have an isomorphism of algebras $I_\B(\be)\simeq I_\Delta(\be)$. The result follows
    since the tensor categories $\B$ and $\phantom{}_\Delta \A_\Delta$ can be 
    reconstructed from the algebras $I_\B(\be)$ and $I_\Delta(\be)$, see \cite[Lemma 8.12.2(ii)]{EGNO}.
\end{proof}

\begin{remark}
    Corollary \ref{8.12.3} appears as \cite[Theorem 8.12.3]{EGNO}. It was pointed out
    by Harshit Yadav that the proof of this result in \cite{EGNO} is incorrect; thus
    Corollary \ref{8.12.3} fills this gap. Note that this new proof uses
    the results of \cite{CSZ} in a crucial way.
\end{remark}

Given an indecomposable commutative exact algebra $A$ in a braided finite tensor category
$\C$ one defines braided finite tensor category $\C_A^{loc}$ of local (or dyslectic) 
$A-$modules, see \cite[2.3]{ShYa} (note that $\C_A^{loc}$ is a full tensor subcategory
of the category $\C_A$). If $\C$ is non-degenerate, then so is $\C_A^{loc}$,
see \cite[Corollary 5.14(b)]{ShYa}.

\begin{corollary}\label{local}
    Let $\A$ be a finite tensor category and let $A\in \cZ(\A)$ be an indecomposable
    commutative exact algebra. Then there exists a finite tensor category $\B$ and
    a braided tensor equivalence $\cZ(\A)_A^{loc}\simeq \cZ(\B)$.
\end{corollary}

\begin{proof}
    Let $A\subset I_\Delta(\be)$ be as in Theorem \ref{lagr}. The action of $A$ on
    $I_\Delta(\be)$ by multiplication makes $I_\Delta(\be)$ into $A-$module; moreover
    $I_\Delta(\be)\in \cZ(\A)_A^{loc}$ and the multiplication in $A$ makes it into
    a commutative algebra in $\cZ(\A)_A^{loc}$, see e.g. \cite[3.6(c)]{DMNO}. 
    This algebra is simple (any ideal of $I_\Delta(\be)\in \cZ(\A)_A^{loc}$ is an ideal
    of $I_\Delta(\be)\in \cZ(\A)$); hence it is exact by Theorem \ref{CoSZ}. 
    It is easy to see that $I_\Delta(\be)$ considered as algebra in $\cZ(\A)_A^{loc}$
    is lagrangian: $\FP(\cZ(\A)_A^{loc})=\frac{\FP(\cZ(\A))}{\FP(A)^2}=\left(\frac{\FP(I_\Delta(\be))}{\FP(A)}\right)^2$ by \cite[Lemma 5.17]{ShYa},
    and $\frac{\FP(I_\Delta(\be))}{\FP(A)}$ is the Frobenius-Perron dimension 
    of $I_\Delta(\be)$ considered as an object of $\cZ(\A)_A^{loc}$. Thus the result
    follows by (a) above.
\end{proof}

\subsection{Witt equivalence}
Let $\C$ and $\D$ be two non-degenerate braided finite tensor categories. One says
that $\C$ and $\D$ are Witt equivalent (denoted by $\C\sim_{Witt}\D$ if there exist
finite tensor categories $\A_1$ and $\A_2$ and an equivalence of braided tensor
categories $\C\boxtimes \cZ(\A_1)\simeq \D \boxtimes \cZ(\A_2)$, see e.g. 
\cite[Definition 7.2]{ShYa} or \cite[Definition 6.23]{LaWa}. 
It is easy to see that $\sim_{Witt}$ is an equivalence relation
and the set of equivalence classes form an abelian group $\W^{ns}$ (with respect to
the operation $\boxtimes$), see \cite[Lemma 7.3]{ShYa}. The identity element 
of the group $\W^{ns}$ is the equivalence class $[\mbox{Vec}]$ where $\mbox{Vec}$ is
the category of finite dimensional vector spaces over $\K$ and $[\C]$ denote the equivalence
class of the category $\C$.

\begin{proposition}\label{witt}
    (1) Assume that $\C\sim_{Witt}\mbox{Vec}$. Then $\C \simeq \cZ(\A)$ for some finite
    tensor category $\A$.

    (2) $\C \sim_{Witt}\D$ if and only if $\C \boxtimes \D^{rev}\simeq \cZ(\A)$
    for some finite tensor category $\A$.
\end{proposition}

\begin{proof}
    (1) The assumption $\C\sim_{Witt}\mbox{Vec}$ says that $\C \boxtimes \cZ(\A_1)\simeq \cZ(\A_2)$ for some finite tensor categories $\A_1$ and $\A_2$. Let $A\in \cZ(\A_1)$
    be a lagrangian algebra. Then $\be \boxtimes A\in \C \boxtimes \cZ(\A_1)$ is an indecomposable commutative exact algebra. Then it is easy to see that
    $(\C \boxtimes \cZ(\A_1))_{\be \boxtimes A}^{loc}\simeq \C$ (e.g. one can use
    exactly the same argument as in \cite[Proposition 5.8]{DMNO}). Let us denote
    by $\tilde A$ the image of $A$ in $\cZ(\A_2)$ under the equivalence
    $\C \boxtimes \cZ(\A_1)\simeq \cZ(\A_2)$. Then by Corollary \ref{local} we have
    $\cZ(\A_2)_{\tilde A}^{loc}\simeq \cZ(\A)$ for some finite tensor category $\A$.
    Hence
    $$\C \simeq (\C \boxtimes \cZ(\A_1))_{\be \boxtimes A}^{loc}\simeq \cZ(\A_2)_{\tilde A}^{loc}\simeq \cZ(\A)$$
    as claimed.

    (2) is immediate from (1) since $\C \sim_{Witt}\D$ is equivalent to $\C \boxtimes \D^{rev}\sim_{Witt}\mbox{Vec}$.
\end{proof}

Proposition \ref{witt} gives positive answers to \cite[Question 7.18]{ShYa} and
\cite[Question 7.21]{ShYa}. As it is explained in \cite[7.4.5]{ShYa} one can use
arguments in the proof of \cite[Theorem 5.13]{DMNO} to prove 

\begin{proposition}
    Every Witt equivalence class contains a unique up to a braided equivalence 
    completely anisotropic representative. Moreover, the completely anisotropic
    representative of the class $[\C]$ is equivalent to $\C_A^{loc}$ where
    $A\in \C$ is an arbitrary maximal indecomposable commutative exact algebra. $\square$
\end{proposition}

Note that if $\C$ is a semisimple category, the category $\C_A$ is also semisimple for
any exact algebra $A$. In particular $\C_A^{loc}$ is semisimple for an indecomposable commutative exact algebra $A$ in a semisimple braided tensor category $\C$. 
Hence we have a negative answer to \cite[Question 7.19]{ShYa}:

\begin{corollary}
    Assume $\C$ is a completely anisotropic non-degenerate braided finite tensor
    category. Then $\C \not \sim_{Witt}\D$ for any semisimple category $\D$.
\end{corollary}

Finally assume $\K=\BC$. Let $\W$ denote the Witt group of non-degenerate
braided fusion categories defined in \cite[5.1]{DMNO}. We have an obvious
homomorphism $\W \to \W^{ns}$.

\begin{theorem}
    The homomorphism $\W \to \W^{ns}$ is injective but not surjective.
    In fact the class $[\C(G_2,G_2(a_1),7,q)]$ is not contained in its image.
\end{theorem}

\section{Conjectures and questions}\label{genconj}
\subsection{Categories $\C(G_2,G_2(a_1),l,q)$}
We expect that Theorem \ref{main 7} extends to the case of arbitrary undivisible (i.e. not
divisible by 3) $l\ge 7$.

\begin{conjecture}\label{g2conj}
    Assume $l\ge 7$ and not divisible by 3.

    (1) We have
    $$\FP(\C(G_2, G_2(a_1),l,q))=6\frac{l^4}{(2\sin(\pi/l))^8(2\sin(2\pi/l))^2},$$
    and 
    $$\FP(T(\omega_1))=2[3]_l+1,\; \FP(T(\omega_2)=[5]_l+3[3]_l.$$

    (2) The category $\C(G_2, G_2(a_1),l,q)$ has stable Chevalley property.

    (3) The M\"uger center of the category $\C(G_2, G_2(a_1),l,q)$ is equivalent
    to $\Rep(S_3)$. The projective covers of the simple objects from $\Rep(S_3)$ are
    in the principal block; the highest weights of the corresponding tilting modules
    are in the alcoves $C_0, C_4, C_7$ (see Figure \ref{fig:subreg_cell_G2}).
\end{conjecture}

We checked that Conjecture \ref{g2conj} holds for $l=7,11$. Interestingly, it also holds
in some sense for $l=5$ if we declare that the category $\C(G_2, G_2(a_1),5,q)$
is a semisimplication of the category $\T(G_2,q)$. It was computed by Etingof
(see also \cite{RW})
that $\C(G_2, G_2(a_1),5,q)$ thus defined is $S_3-$equivariantization of a product
of 3 copies of the Fibonacci's category and this is compatible with formula
in Conjecture \ref{g2conj} (1). 

In the case when $l$ is divisible by 3 we still expect that the category
$\C(G_2, G_2(a_1),l,q)$ contains a subcategory $\Rep(S_3)$; however the M\"uger center
might be strictly contained in this subcategory. We still expect the stable Chevalley property and have a conjectural formula
for the Frobenius-Perron dimension:

\begin{conjecture}\label{g2conjm}
    Assume $l\ge 12$ and $l$ is divisible by 3. Then

    $$\FP(\C(G_2, G_2(a_1),l,q))=\frac29\frac{l^4}{(2\sin(\pi/l))^4(2\sin(2\pi/l))^2(2\sin(3\pi/l))^4}.$$
and for $l\ge 15$
$$\FP(T(\omega_1))=[3]_l+[5]_l-1,\; \FP(T(\omega_2))=2[7]_l-[5]_l+[3]_l+2.$$
\end{conjecture}

We checked that the formulas work for $l=12,15,18,21$; this applies also for $l=9$ if we
define category $\C(G_2, G_2(a_1),9,q)$ to be a semisimplification of $\T(\g,q)$
(for $q$ such that $q^2$ has order 9), see \cite{RW}. The formula for $\FP(T(\omega_2))$ gives
an incorrect result for $l=12$, namely the right hand side and the left hand side
differ exactly by 1. This is likely to be explained by the fact that $T(\omega_2)$
does not coincide with Weyl module in this (and only in this) case.

\subsection{Categories $\C(\g,e,l,q)$ for distinguished nilpotent $e\in \g$}\label{distconj}
Let $Q=Q(e)$ be the centralizer of a $sl_2-$triple associated with $e$ in the simply
connected group $G$ with $\mbox{Lie}(G)=\g$. In the case of distinguished nilpotent $e$,
$Q$ is a finite group. Let $2k_i+1, i\in M=M(e)$ be the sizes of Jordan cells for
the adjoint action of $e$ on $\g$ (it is well known that these sizes are odd for distinguished
nilpotent elements; the number of these cells, that is the cardinality of set $M$, is
the dimension of the centralizer of $e$ in $\g$). For any $k\in \Z_{>0}$ we set
$$S_k(l):=\frac{l}{(2^k\sin(\pi/l)\sin(2\pi/l)\cdots \sin(k\pi/l))^2}.$$

\begin{conjecture}\label{FPdim}
    Assume $m=1$ or $l$ is undivisible. Then
    $$\FP(\C(\g,e,l,q))=|Q|\prod_{i=1}^MS_{k_i}(l).$$
\end{conjecture}

In the case of Lie algebras $\g$ of exceptional type the numbers $k_i$ can be found
in \cite{St}; in the case of classical Lie algebras one computes them easily from
partitions associated to $e$, see \cite{CMg}. 

\begin{example}
    (0) Assume $\g$ is of type $G_2$ and $e$ is of type $G_2(a_1)$. Then $\g$ decomposes
    into Jordan cells of sizes $5,3,3,3$, see \cite[Table 16]{St}. Also $Q=S_3$, so Conjecture \ref{FPdim} predicts $\FP(\C(G_2, G_2(a_1),l,q))=6S_2(l)S_1(l)^3$ which coincides
    with formula in Conjecture \ref{g2conj} (1).

    (1) Assume $\g$ is of type $C_3$ and $e$ is the regular element. Then $\g=sp(6)$
    and the partition associated to $e$ is $[6]$. Thus the tautological representation $W$
    of $\g$ is the irreducible 6-dimensional representation of the correponding $sl_2-$triple.
    Hence $\g \simeq S^2W$ is a direct sum of 11, 7, and 3-dimensional $sl_2-$representations,
    so $k_i=5,3,1$. Also $Q=\Z/2$ in this case. Thus Conjecture \ref{FPdim} predicts 
    $$\FP(\C(C_3,l,q))=2S_{5}(l)S_3(l)S_1(l)=$$ $$2\frac{l^3}{(2\sin(\pi/l))^6(2\sin(2\pi/l))^4(2\sin(3\pi/l))^4(2\sin(4\pi/l))^2(2\sin(5\pi/l))^2}.$$
    We did not find a formula for $\FP(\C(\g,l,q))$ in the undivisible case (for $m>1$)
    in the literature, so the formula above is a conjecture.

    (2) Take again $\g$ of type $C_3$ but let $e$ be the subregular nilpotent element.
    The associated partition is $[4,2]$; so $W$ is a direct sum of 2-dimensional and 4-dimensional irreducible $sl_2-$representations. Hence $\g\simeq S^2W$ is a direct sum
    of 7, 5, and three 3-dimensional $sl_2-$representations. Hence $k_i=3,2,1,1,1$. Also
    $Q\simeq \Z/2\times \Z/2$. Thus Conjecture \ref{FPdim} predicts
    $$\FP(\C(C_3,e,l,q))=4S_{3}(l)S_2(l)S_1(l)^3=4\frac{l^5}{(2\sin(\pi/l))^{10}(2\sin(2\pi/l))^4(2\sin(3\pi/l))^2}.$$
    
\end{example}

Here is a heuristic explanation of Conjecture \ref{FPdim}. Assume that $l$ is a prime.
We expect that the combinatorics (and, in particular the Frobenius-Perron dimension) of the category $\C(\g,e,l,q)$ coincides with the combinatorics of a similar category $\C(\g,e,l,q)_l$
defined over a field of characteristic $l$ (in this case $q=1$ and the category $\T(\g,q)$
is the category of tilting modules over the algebraic group $G$). The category
$\C(\g,e,l,q)_l$ is a symmetric tensor category; we expect that it admits 
an exact tensor functor $\C(\g,e,l,q)_l\to \mbox{Ver}_l$ where $\mbox{Ver}_l$
is the Verlinde category (see e.g. \cite{BEO}). Moreover, on the subcategory of
tilting modules in $\C(\g,e,l,q)_l$ this functor is isomorphic to the following
obvious functor: restrict tilting module to the group scheme $(\G_a)_1$ (the Frobenius 
kernel of the additive group $\G_a$) spanned by $e$ and semisimplify. Now the
Frobenius-Perron dimensions of objects can be computed in the category $\mbox{Ver}_l$;
thus we can compute Frobenius-Perron dimensions of tilting modules just by
looking at the action of $e$ on such modules; e.g. this gives formulas for the Frobenius-Perron dimensions of the fundamental representations as in Conjecture \ref{g2conj} (1). On the other hand the Frobenius-Perron dimension of the
category $\C(\g,e,l,q)_l$ is equal to the Frobenius-Perron dimension of its
fundamental group; we expect the image of this group in $\mbox{Ver}_l$ to have $Q$
as its group of components and the image of $\g$ as the Lie algebra of its
infinitesimal part; in particular, its dimension is a product of $|Q|$ and the
dimension of the universal enveloping algebra of $\g \in \mbox{Ver}_l$. By
PBW theorem (see e.g. \cite{Epbw}) the universal enveloping algebra has the same
dimension as the symmetric algebra $S(\g)$. Finally, an elementary calculation 
shows that $\FP(S(L_{2k+1}))=S_k(l)$ and we arrive at the formula in Conjecture \ref{FPdim}.

We expect that Conjecture \ref{FPdim} can be extended to the divisible case
similarly to Conjecture \ref{g2conjm} for type $G_2$. However, the precise
statement in this case remains to be found.

It is a classical result of P.~Slodowy that the singularity of the nilpotent
cone of $\g$ at the point $e_{sr}$ is $\BC^2/\Gamma$ for a suitable subgroup
$\Gamma \subset SL_2(\BC)$, see \cite{Sl}. Comparing this with Theorem \ref{main blocks}
we expect that the cohomology of the category $\C(\g,e,l,q)$ is related
with the singularity of the nilpotent cone at point $e$. Thus let $\Sl_e$
be the Slodowy slice at the point $e$, see e.g. \cite{GG} (thus $\Sl_e$ is the intersection
of the affine space appearing in \cite[1.1]{GG} and the nilpotent cone). The variety $\Sl_e$
is equipped with an action of $Q\times \BC^\times$. Thus, the algebra of functions
$\mO(\Sl_e)$ is graded and is equipped with a grading-preserving action of $Q$;
let $\mO(\Sl_e)^Q\subset \mO(\Sl_e)$ be the algebra of invariants.

\begin{conjecture}\label{Slod}
    We have an isomorphism of graded algebras
$$\mbox{Ext}_{\C(\g,e,l,q)}(\be,\be)=\mO(\Sl_e)^Q.$$
\end{conjecture}

\begin{example}
    Assume $\g$ is of type $G_2$ and $e=e_{sr}$ is the subregular nilpotent element.
    Theorem \ref{main blocks} states that $\mbox{Ext}_{\C(\g,e,l,q)}(\be,\be)=S^\bullet(V)^\Gamma$ where $V$ is a two dimensional space and $\Gamma \subset SL(V)$ is subgroup with the McKay graph of type $\tilde E_7$ (i.e. $\Gamma$ is the binary octahedral group). It is well known that there is a surjective homomorphism
    $\Gamma \to S_3$ with the kernel $\Gamma_1$ isomorphic to the group of quaternions
    (so the McKay quiver of $\Gamma_1$ is of type $\tilde D_4$). Recall that in this case
    $Q=S_3$ and $\Sl_e\simeq \BC^2/\Gamma_1$, see \cite{Sl}. Thus, in this case,
    $$\mbox{Ext}_{\C(\g,e,l,q)}(\be,\be)=S^\bullet(V)^\Gamma=(S^\bullet(V)^{\Gamma_1})^{S_3}=\mO(\Sl_e)^Q.$$
    This confirms Conjecture \ref{Slod} in this case (modulo an identification of the
    action of $Q=S_3$ on $\Sl_e$ with the action of $S_3$ on $\BC^2/\Gamma_1=\Sl_e$).
\end{example}

It would be interesting to find a description of the principal block of
the category $\C(\g,e,l,q)$. For starters one can ask

\begin{question}
    What is the Cartan matrix of the principal block of the category $\C(\g,e,l,q)$?
\end{question}

Also it seems interesting to find representation theoretic realization
of the category $\C(\g,e,l,q)$.

\begin{question}
    Is it possible to realize $\C(\g,e,l,q)$ as representation category of some
    vertex algebra?
\end{question}

Here is another question:

\begin{question}
    What can be said about semisimplification of the category $\C(\g,e,l,q)$?
\end{question}

For general $e$ we expect the category $\C(\g,e,l,q)$ to be of wild representation
type, so its semisimplification is unlikely to be computed completely. However in the subregular case we expect that that the semisimisimplification 
of $\C(\g,e_{sr},l,q)$ is $\C(\g,l,q)\boxtimes \Rep(\G_m\times \Gamma,\varepsilon)$
where $\varepsilon =(-1,\epsilon)$ (an element $\epsilon \in \Gamma$ was defined
in the beginning of Section \ref{toy}). Here $\C(\g,l,q)$ comes from the semisimplification of the tilting modules inside of $\C(\g,e_{sr},l,q)$,
$\Rep(\G_m)=\mbox{Vec}_\Z$ comes from the Heller shifts of the unit object,
and $\Rep(\Gamma)$ factor comes from the simple objects in the principal block.

\subsection{Categories $\C(\g,e,l,q)$ for general nilpotent $e$}
We expect that the categories $\C(\g,e,l,q)$ are defined for arbitrary nilpotent $e$,
that is the quotient $\T(\g,q)/\cI_e$ admits an abelian monoidal envelope.
However, the category $\C(\g,e,l,q)$ is not finite when $e$ is not distinguished.
Nevertheless, the category $\C(\g,e,l,q)$ is an equivariantization of some
finite tensor category $\tilde \C(\g,e,l,q)$ by a group closely related to $Q$
(e.g. the quotient of $Q$ by a normal subgroup of order 2).

The only example we know is the case $e=0$. Then $\cI_e=0$ and the minimal thick ideal of the category $\T(\g,q)$ is the category of projective
$U_q-$modules. Thus, the category $\C(\g,0,l,q)$ is the category
of finite dimensional $U_q-$modules; by the existence of the quantum
Frobenius map (see \cite{Lb}) this category is $G-$eqivariantization
of the category of representations of the small quantum group.

\bibliographystyle{amsalpha}

\begin{thebibliography}{A}

\bibitem[A]{An} H.~H.~Andersen, Tensor products of quantized tilting modules, Comm. Math. Phys. 149 (1992), no. 1, 149--159.

\bibitem[AP]{AP} H.~H.~Andersen, J.~Paradowski, Fusion categories arising from semisimple Lie algebras, 
Comm. Math. Phys. 169 (1995), no. 3, 563--588.

\bibitem[BK]{BK} B.~Bakalov, A.~Kirillov Jr., \textit{Lectures on tensor categories and modular functors}, University Lecture Series, 21. American Mathematical Society, Providence, RI, 2001.

\bibitem[BEO]{BEO} D.~Benson, P.~Etingof, V.~Ostrik, New incompressible symmetric tensor categories in positive characteristic. Duke Math. J. 172 (2023), no. 1, 105--200. 

\bibitem[B]{Brug} A.~Brugui\`eres, Cat\'egories pr\'emodulaires, modularisations et invariants des vari\'et\'es de dimension 3, Math. Ann. 316 (2000), no. 2, 215--236.

\bibitem[CM]{CMg} D.~Collingwood, W.~McGovern, \textit{Nilpotent orbits in semisimple Lie algebras}, Van Nostrand Reinhold Mathematics Series. Van Nostrand Reinhold Co., New York, 1993.

\bibitem[C]{Co} K.~Coulembier, Tensor ideals, Deligne categories and invariant theory. Selecta Math. (N.S.) 24 (2018), no. 5, 4659--4710.

\bibitem[CEO]{CEO} K.~Coulembier, P.~Etingof, V.~Ostrik, Tensor ideals of abelian type and quantum groups,
arXiv: 2511.08859.

\bibitem[CSZ]{CSZ} K.~Coulembier, M.~Stroi\'nski, T.~Zoran, Simple algebras and exact module categories,
arXiv: 2501.06629.

\bibitem[D]{D} E.~C.~Dade, Group-graded rings and modules,
Math. Z. 174 (1980), no. 3, 241--262.

\bibitem[DMNO]{DMNO} A.~Davydov, M.~M\"uger, D.~Nikshych, V.~Ostrik, The Witt group of non-degenerate braided fusion categories, J. f\"ur die reine
und angewandte Mathematik, {\bf 677} (2013), p. 135-177.

\bibitem[DNO]{DNO} A.~Davydov, D.~Nikshych, V.~Ostrik, On the structure of the Witt group of braided fusion categories. Selecta Math. (N.S.) 19 (2013), no. 1, 237--269. 

\bibitem[DGNO]{DGNO} V. ~Drinfeld, S. ~Gelaki, D. ~Nikshych, and V. ~Ostrik,
 On Braided Fusion Categories I. Selecta Math. 16 (2010), no. 1, 1--119.

\bibitem[EGNO]{EGNO} 
P.~Etingof, S.~Gelaki, D.~Nikshych, and V.~Ostrik,
\textit{Tensor categories}, Mathematical Surveys and Monographs, 205. American Mathematical Society, Providence, RI, 2015.

\bibitem[E1]{Ev} P.~Etingof, On Vafa's theorem for tensor categories,
Math. Res. Lett. 9 (2002), no. 5-6, 651--657.

\bibitem[E2]{Epbw} P.~Etingof, Koszul duality and the PBW theorem in symmetric tensor categories in positive characteristic,
Adv. Math. 327 (2018), 128--160.

\bibitem[EO1]{EOfinite} P.~Etingof, V.~Ostrik, Finite tensor categories, Moscow Math. J. 
{\bf 4} (2004), no. 3, p. 627-654.

\bibitem[EO2]{EOsemi} P.~Etingof, V.~Ostrik, On semisimplification of tensor categories. Representation theory and algebraic geometry--a conference celebrating the birthdays of Sasha Beilinson and Victor Ginzburg, 3--35, Trends Math., Birkh\"auser/Springer, Cham, 2022.

\bibitem[GG]{GG} W.~L.~Gan, V.~Ginzburg, Quantization of Slodowy slices, Internat. Math. Res. Notices 5 (2002) 243--255.

\bibitem[KT]{KaTa} M.~Kashiwara, T.~Tanisaki, Kazhdan-Lusztig conjecture for symmetrizable Kac-Moody Lie algebras. III. Positive rational case, Asian J. Math. 2 (1998), no. 4, 779--832.

\bibitem[LW]{LaWa} R.~Laugwitz, C.~Walton, Constructing non-semisimple modular categories with local modules. Comm. Math. Phys. 403 (2023), no. 3, 1363--1409.

\bibitem[L1]{Luex} G.~Lusztig, Some examples of square integrable representations of semisimple p-adic groups, Trans. Amer. Math. Soc. 277 (1983), 623--653.

\bibitem[L]{Lucells4} G.~Lusztig, Cells in affine Weyl groups. IV. J. Fac. Sci. Univ. Tokyo Sect. IA Math. 36 (1989), no. 2, 297--328.

\bibitem[L2]{Lb} G.~Lusztig, \textit{Introduction to quantum groups}, Progress in Mathematics, 110. Birkh\"auser Boston, Inc., Boston, MA, 1993.

\bibitem[LX]{LX} G.~Lusztig, N.~Xi, Canonical left cells in affine Weyl groups. Adv. in Math. 72 (1988), no. 2, 284--288.

\bibitem[NO]{NO} C.~N\u{a}st\u{a}sescu, F.~van Oystaeyen, \textit{Graded ring theory}, North-Holland Mathematical Library, 28. North-Holland Publishing Co., Amsterdam-New York, 1982.

\bibitem[NC]{cafe} nLab authors,
  monoidal adjunction,
  
 {\url{https://ncatlab.org/nlab/show/monoidal+adjunction}},
 {\href{https://ncatlab.org/nlab/revision/monoidal+adjunction/12}{Revision 12}}, January, 2026.


\bibitem[O]{Ost1} V.~Ostrik, Tensor ideals in the category of tilting modules,
Transformation Groups, {\bf 2} (1997), no. 3, p. 279-287.

\bibitem[OU]{OU} V.~Ostrik, A.~Utiralova, Python code used in this paper,
{\url{https://github.com/autiralova/A-non-semisimple-Witt-class-code}}.

\bibitem[R]{Rasm} T.~E.~Rasmussen, Multiplicities of second cell tilting modules. J. Algebra 288 (2005), no. 1, 1--19.

\bibitem[Ri]{Ri} C.~M.~Ringel, The indecomposable representations of the dihedral 2-groups. Math. Ann. 214 (1975), 19--34.

\bibitem[RW]{RW} E.~Rowell, H.~Wenzl, Fusion categories of type $G_2$ at levels $-8/3,-7/3$ and $-1$, preprint.

\bibitem[S]{Saw} S.~Sawin, Quantum groups at roots of unity and modularity, J. Knot Theory Ramifications 15 (2006), no. 10, 1245--1277.

\bibitem[Sh]{Shim} K.~Shimizu, Non-degeneracy conditions for braided finite tensor categories, Adv. Math. 355 (2019), 106778, 36 pp.

\bibitem[SY]{ShYa} K.~Shimizu, H.~Yadav, Commutative exact algebras and modular tensor categories, arXiv:2408.06314. 

\bibitem[Sc]{Scho} A.~Schopieray, Lie theory for fusion categories: a research primer, Topological phases of matter and quantum computation, 1--26, Contemp. Math., 747, Amer. Math. Soc., Providence, RI, 2020.

\bibitem[Sl]{Sl} P.~Slodowy, {\em Simple Singularities and Simple Algebraic Groups}, Lecture Notes in Mathematics, vol. 815, Springer, Berlin, 1980.

\bibitem[S1]{So1} W.~Soergel, Kazhdan-Lusztig-Polynome und eine Kombinatorik f\"ur Kipp-Moduln, 
Represent. Theory 1 (1997), 37--68.

\bibitem[S2]{So2} W.~Soergel, Charakterformeln f\"ur Kipp-Moduln \"uber Kac-Moody-Algebren, 
Represent. Theory 1 (1997), 115--132.

\bibitem[St]{St} D.~I.~Stewart, On the minimal modules for exceptional Lie algebras: Jordan blocks and stabilizers, LMS J. Comput. Math. 19 (2016), no. 1, 235--258.

\end{thebibliography}

\end{document}